\documentclass[12pt]{article}

\usepackage{color}
\usepackage{amsfonts}
\usepackage{amsmath}
\usepackage{amssymb}
\usepackage{graphicx}
\usepackage{mathptmx}
\usepackage{subfigure}
\usepackage{float}
\usepackage{parskip}
\RequirePackage[colorlinks, hyperindex, plainpages=false]{hyperref}
\usepackage{amsmath}

\def\eps{\varepsilon}

\def\P{{\mathbb P}}
\def\E{{\mathbb E}}

\def\R{{\mathbb R}}

\def\1{{\mathbf 1}}

\newcommand{\FF}          {\mathcal{F}}

\newcommand{\HH}         {\mathcal{H}}

\newtheorem{theorem}{Theorem}[section]
\newtheorem{lemma}[theorem]{Lemma}
\newtheorem{corollary}[theorem]{Corollary}
\newtheorem{proposition}[theorem]{Proposition}

\newtheorem{remark}[theorem]{Remark}

\newenvironment{proof}[1][Proof.]{\textbf{#1} }{\hfill $\blacksquare$}
\def\beq{\begin{equation}}
\def\eeq{\end{equation}}
\newcommand{\bei}{\begin{itemize}}
\newcommand{\eei}{\end{itemize}}
\newcommand{\ben}{\begin{enumerate}}
\newcommand{\een}{\end{enumerate}}
\newcommand{\beqn}{\begin{eqnarray}}
\newcommand{\beqnn}{\begin{eqnarray*}}
\newcommand{\eeqn}{\end{eqnarray}}
\newcommand{\eeqnn}{\end{eqnarray*}}
\newcommand{\brm}{\begin{rmk}}
\newcommand{\erm}{\end{rmk}}

\begin{document}

\title{Approximation of the Height process of a continuous state branching process with interaction}

\author{
I.~Dram\'{e}  \footnote{  \scriptsize{Université Cheikh Anta Diop de Dakar, FST, LMA, 16180 Dakar-Fann, S\'en\'egal. iboudrame87@gmail.com
}}
\and
E.~Pardoux \setcounter{footnote}{6}\footnote{ \scriptsize {Aix-Marseille Universit\'{e}, CNRS,  Centrale Marseille, I2M, UMR 7373, 13453 Marseille, France. etienne.pardoux@univ-amu.fr}}
}

\maketitle

\begin{abstract}
In this work, we first show that the properly rescaled height process of the genealogical tree of a continuous time branching process converges to the height process of the genealogy of a (possibly discontinuous) continuous state branching process. We then prove the same type of result for generalized branching processes with interaction.
\end{abstract}

\vskip 3mm

\section{Introduction}
Continuous state branching processes (or CSBP in short) are the analogues of Galton-Watson (G-W) processes in continuous time and continuous state space. Such classes of processes have been introduced by Jirina \cite{Ji} and studied by many authors included Grey \cite{Grey}, Lamperti \cite{Lam1},  to name but a few. These processes are the only possible weak limits that can be obtained from sequences of rescaled G-W processes, see  Lamperti \cite{Lam2} and Li \cite{Li1}, \cite{Li2}. 

While rescaled discrete-time G-W processes converge to a CSBP,  it has been shown in Duquesne and Le Gall \cite{DLG} that the genealogical structure of the G-W processes converges too. More precisely, the corresponding rescaled sequences of discrete height process, converges to the height process in continuous time that has been introduced by Le Gall and Le Jan in \cite{lGlJ}.

A lot of work has been devoted recently to generalized branching processes, which model competition within the population. This includes generalized CSBPs, see among many others Li \cite{Lps}, Li, Yang and Zhou \cite{LYZ} and the references therein. For the approximation of such generalized CSBPs by discrete time generalized GW processes, we refer to the general results in Bansaye, Caballero and M\'el\'eard \cite{BCM}, and for the approximation by continuous time 
generalized GW processes to our recent paper \cite{DP}.

Some work has been also devoted recently to the description of the genealogy of such generalized CSBPs,
see Le, Pardoux and Wakolbinger \cite{LePW} and Pardoux \cite{EP} for the case of continuous such processes and 
both Berestycki, Fittipaldi and Fontbona \cite{BFF} and Li, Pardoux and Wakolbinger \cite{LiPW} for the general case.
\cite{BFF} allows processes without a Brownian component unlike \cite{LiPW}, but the latter allows more general interactions.
The present paper studies the convergence of the genealogy of a generalized continuous time GW process to that of a generalized possibly discontinuous CSBP, under the same assumptions as \cite{LiPW}. 

We first give a construction of the CSBP as a scaling limit of continuous time G-W branching processes. To then give a precise meaning to the convergence of trees, we will code G-W trees by a continuous exploration process as already defined by Dram\'{e} et al. in \cite{DPS}, and we will establish the convergence of this (rescaled) continuous process to the continuous height process defined in \cite{LiPW}, see also  \cite{DLG}. Each jump of our generalized CSBP corresponds to the birth of a significant proportion of the total population, whose genealogical tree needs to be explored by our height process. This gives rise to a special term in the equation for the height process, which has possibly unbounded variations and has no martingale property. It also destroys any possible Markov property of the height process. The tightness of such a term cannot be established by standard techniques. We use for that purpose a special method which has been developed in \cite{LiPW}, see the proof of Proposition \ref{convHN} below. The main result of this paper is Theorem \ref{thfin} in section \ref{sec4}.

The organization of the  paper is as follows : In Section 2 we recall some basic definitions and notions concerning branching processes. 
Section 3, which is by far the longest one, considers the height process in the case without interaction. It is devoted to the description of the discrete approximation of both the population process and the height process of its genealogical tree. We prove the convergence of the height process, and of its local time. Section 4 introduces the interaction, via a Girsanov change of probability measure, and establishes the main result. We consider first the case where the interaction function has a bounded derivative, and then the general situation, which allows in particular the popular so--called logistic (i.e. quadratic) interaction.

 We shall assume that all random variables in the paper are defined on the same probability space $( \Omega,\mathcal{F},\mathbb{P})$. We shall use the following notations $\mathbb{Z}_+=\{0,1, 2,... \}$, $\mathbb{N}=\{1, 2,... \}$, $\mathbb{R}=(-\infty, \infty)$ and $\mathbb{R}_+=[0, \infty)$. For $x$ $\in$ $\mathbb{R}_+$, $[x]$ denotes the integer part of $x$.

\section{The Height process of a continuous state branching process}
\subsection{Continuous state branching process}
A CSBP is a $\mathbb{R}_+$-valued strong Markov process has the property, $\mathbb{P}_x$ denoting the law of the process when starts from $x$ at time $t=0$, $\mathbb{P}_{x+y}= \mathbb{P}_{x} \ast \mathbb{P}_{y}$. 
More precisely, a CSBP $X^x=(X_t^x, \ t\geqslant0)$ (with initial condition $X_0^x=x$) is a Markov process taking values in $[0, \infty]$, where $0$ and $\infty$ are two absorbing states, and satisfying the branching property; that is to say, it's Laplace transform satisfies 
\begin{equation*}
   \mathbb{E} \left[ \exp (-\lambda X_t^x)\right] = \exp \left\{ -x u_t(\lambda)\right\}, \quad \mbox{for} \ \lambda\geqslant0, 
\end{equation*}
for some non negative function $u_t(\lambda)$. According to Silverstein \cite{Sil}, the function $u_t$ is the unique nonnegative solution of the integral equation
\begin{equation}\label{Utlamda}
 u_t(\lambda) = \lambda - \int_{0}^{t} \psi ( u_r(\lambda)) dr, 
\end{equation}
where $\psi$ is called the branching mechanism associated with $X^x$ and is defined by
\begin{equation*}
 \psi(\lambda) = b\lambda +{c\lambda^2}+ \int_{0}^{\infty} (e^{-\lambda z}-1+\lambda z\mathbf{1}_{\{z\le 1\}}) \mu(dz), 
\end{equation*}
where $b \in \mathbb{R}$, $c\geqslant0$ and $\mu$ is a $\sigma$-finite measure which satisfies $(1\wedge z^2) \mu(dz)$ is a finite measure on $(0,\infty)$. We shall in fact assume in this paper that 
\begin{equation*}
 {(\bf H)}  : \quad \int_{0}^{\infty} (z\wedge z^2) \mu(dz) < \infty \quad \mbox{and} \quad c>0 .
\end{equation*}
The first assumption implies that the process $X^x$ does not explode and we allows to write the last integral in the above equation in the following form 
\begin{equation}\label{PSIEXPRES}
 \psi(\lambda) = b\lambda +{c\lambda^2}+ \int_{0}^{\infty} (e^{-\lambda z}-1+\lambda z) \mu(dz).
\end{equation}
Let us recall that $b$ represents a drift term, $c$ is a diffusion coefficient and $\mu$ describes the jumps of the CSBP. The CSBP is then characterized by the triplet $(b, c, \mu)$ and can also be defined as the unique non negative strong solution of a stochastic differential equation. More precisely, from Fu and Li \cite{FL} (see also the results in Dawson-Li \cite{DaLi}) we have 
\begin{equation}\label{XT}
X_t^x= x - b \int_{0}^{t} X_s^x ds + \sqrt{2c} \int_{0}^{t}  \int_{0}^{X_s^x} W(ds,du) + \int_{0}^{t}\int_{0}^{\infty}\int_{0}^{X_{s^-}^x}z \overline{M}(ds, dz, du),
\end{equation}
where $W(ds,du)$ is a space-time white nose on $(0,\infty)^{2}$, $M(ds, dz, du)$ is a Poisson random measure on $(0,\infty)^{3}$, with intensity $ds\mu(dz)du$, and $\overline{M}$ is the compensated measure of $M$. 
\subsection{The height process in the case without interaction}
We shall also interpret below the function $\psi$ defined by \eqref{PSIEXPRES} as the Laplace exponent of a spectrally positive L\'{e}vy process $Y$. Lamperti \cite{Lam1} observed that CSBPs are connected to L\'{e}vy processes with no negative jumps by a simple time-change. More precisely, define
\begin{equation*}
A_s^x=\int_{0}^{s}X_t^x dt, \quad \tau_s=\inf\{t>0, \ A_t^x >s\} \quad \mbox{and} \quad Y(s)=X_{\tau_s}^x. 
\end{equation*}
Then $Y(s)$ is a L\'evy process of the form until the first that it hits $0$
\begin{equation}\label{YLevy}
Y(s)= - b s + \sqrt{2c} B(s) + \int_{0}^{s}\int_{0}^{\infty}z \overline{\Pi}(dr, dz),
\end{equation}
where $B$ is a standard Brownian motion and $\overline{\Pi}(ds, dz)= \Pi(ds, dz)- ds\mu(dz)$, $\Pi$ being a Poisson random measure on $\mathbb{R}_{+}^{2}$ independent of $B$ with mean measure $ds\mu(dz)$. We refer the reader to \cite{Lam1} for a proof of that result.
To code the genealogy of the CSBP, Le Gall and Le Jan \cite{lGlJ} introduced the so-called height process, which is a functional of a L\'{e}vy process with Laplace exponent $\psi$; see also Duquesne and Le Gall \cite{DLG}. In this paper, we will use the new definition of the height process $H$ given by Li et all in \cite{LiPW}. Indeed, if the L\'{e}vy process $Y$ has the form \eqref{YLevy}, then the associated height process is given by 
\begin{equation}\label{CHS1}
c H(s)= Y(s) - \inf_{0\leqslant r \leqslant s} Y(r) - \int_{0}^{s}\int_{0}^{\infty}\left(z + \inf_{r\leqslant u \leqslant s}Y(u)-Y(r) \right)^+\Pi(dr, dz),
\end{equation}
and it has a continuous modification. Note that the height process $H$ is the one defined in Chapter 1 of  \cite{DLG}.
\begin{figure}[H]
\centering
\includegraphics[width=3.5in]{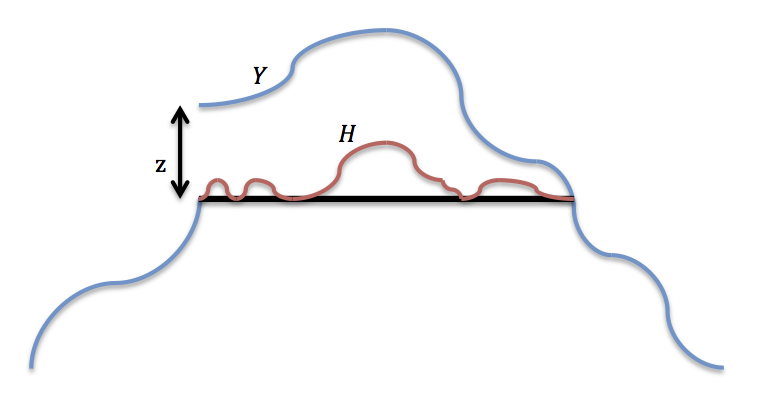}
\caption{Trajectories of $Y$ and $H$.}
\end{figure} 
Let $L_s(t)$ denote the local time accumulated by the process $H$ at level $t$ up to time $s$. The existence of $L_s(t)$ was established already in \cite{DLG}. We have the following Proposition, which can be found in Li et all in \cite{LiPW}.  
\begin{proposition}\label{Tanaka} 
(It\^{o}-Tanaka formula for the local time of $H$) We have 
\begin{equation}\label{solB}
L_s(t)= c(H(s)-t)^+ -  \int_{0}^{s} \mathbf{1}_{\{H(r)>t\}} dY(r) + \int_{0}^{s}\int_{0}^{\infty}\mathbf{1}_{\{H(r)>t\}} \left(z + \inf_{r\leqslant u \leqslant s}Y(u)-Y(r) \right)^+\Pi(dr, dz).
\end{equation}
\end{proposition} 

\subsection{The height process in the case with interaction}
Now the stochastic differential equation \eqref{XT} is replaced by  
\begin{equation}\label{XTPLUS}
X_t^x= X_0^x + \int_{0}^{t} f(X_s^x) ds + \sqrt{2c} \int_{0}^{t}  \int_{0}^{X_s^x} W(ds,du)  + \int_{0}^{t}\int_{0}^{\infty}\int_{0}^{X_{s^-}^x}z \overline{M}(ds, dz, du),
\end{equation}
where $f$ is a function $f : \mathbb{R}^+ \rightarrow \mathbb{R}$, which satisfies 
\begin{equation}\label{hypof}
f \in \mathcal{C}^1(\mathbb{R}^+ ), \quad f(0)=0, \quad f^\prime(z) \le \theta.
\end{equation}
for all  $z \in \mathbb{R}$, for some $\theta \in \mathbb{R}$. In this case, the process $Y$ will be defined as 
\begin{equation}\label{YLEVYR}
Y(s)= \sqrt{2c} B(s) + \int_{0}^{s}\int_{0}^{\infty}z \overline{\Pi}(dr, dz),
\end{equation}
where $B$ is again standard Brownian motion and again $\overline{\Pi}$ denotes the compensated measure  $\overline{\Pi}(ds, dz)= \Pi(ds, dz)- ds\mu(dz).$ This means that in this subsection $b=0$.

The SDE for $H$ reads, see \cite{LiPW}  
\begin{align}\label{CHS12}
c H(s)&= Y(s) + L_s(0) +\int_{0}^{s} f^{\prime}(L_{r}(H(r)))dr \nonumber\\
&- \int_{0}^{s}\int_{0}^{\infty}\left[ z - (L_{s}(H(r))-L_{r}(H(r))) \right]^+\Pi(dr, dz).
\end{align}

\section{Approximation of the Height process without interaction}
In the following,  we consider a specific forest of Bellman-Harris trees, obtained by Poissonian sampling of the height process $H$. In other words, let $\alpha >0$ and we consider a standard Poisson process with intensity $\alpha$. We denote by $\tau_1^\alpha \le \tau_2^\alpha \le \cdots $ the jump times of this Poisson process. If $H$ is seen as the contour process of a continuous tree, consider the forest of the smaller trees carried by the jump times $\tau_1^\alpha \le \tau_2^\alpha \le \cdots $. We have the following proposition, which can be found in \cite{DLG}.

\begin{proposition} $(\mathbf{Duquesne \ and \ Le\ Gall} \ 2002, \ Theorem \ 3.2.1 )$
The trees in this forest are trees, which are distributed as the family tree of a continuous-time Galton-Watson process starting with one individual at time $0$ and such that : 

$\ast$ Lifetimes of individuals have exponential distributions with parameter  $\psi^{\prime} ( \psi^{-1}(\alpha))$;

$\ast$ The offspring distribution is the law of the variable $\eta$ with generating function
\[ \E(s^\eta)= s+ \frac{\psi((1-s)\psi^{-1}(\alpha))}{\psi^{-1}(\alpha) \psi^{\prime} ( \psi^{-1}(\alpha))}.\]
\end{proposition}

Let $N\geqslant1$ be an integer which will eventually go to infinity. In the next two sections, we choose a sequence $\delta_N \downarrow 0$ such that, as $N\rightarrow \infty$, 
\begin{equation*}
{(\bf A)}  : \quad \frac{1}{N} \int_{\delta_N}^{+\infty}  \mu(dz) \rightarrow 0.
\end{equation*}
This implies in particular that 
\begin{equation*}
\frac{1}{N} \int_{\delta_N}^{+\infty} z \mu(dz) \rightarrow 0.
\end{equation*}
Moreover, we will need to consider
\begin{equation}\label{deltan}
 \psi_{\delta_N}(\lambda) =  {c\lambda^2}+ \int_{\delta_N}^{\infty} (e^{-\lambda z}-1+\lambda z) \mu(dz).
\end{equation}
We will also set $\alpha= \psi_{\delta_N}(N)$ in the limit of large populations. 

\subsection{A discrete mass approximation}\label{sec:GWcont}

In this subsection, we obtain a CSBP as a scaling limit of continuous time Galton-Watson branching processes. In other words, the aim of this subsection is to set up a "discrete mass - continuous time" approximation of \eqref{XT} . To this end, we set 
\begin{equation}\label{genepi}
h_{N}(s)= s + \frac{\psi_{\delta_N}((1-s)N)}{N \psi_{\delta_N}^{\prime} (N)}, \quad \quad |s| \leqslant 1.
\end{equation}
It is easy to see that $s\rightarrow h_{N}(s)$ is an analytic function in $(-1,1)$ satisfying $h_{N}(1)=1$ and 
\begin{equation*}
\frac{d^n}{ds^n}h_{N}(0)\ge 0, \quad  n \ge 0.
\end{equation*}
Therefore $h_{N}$ is a probability generating function. and we have
\begin{equation*}
h_{N}(s)= \sum_{\ell\ge0} \nu_{N}(\ell)s^\ell,  \quad |s| \leqslant 1,
\end{equation*}
where $\nu_{N}$ is probability measure on $\mathbb{Z}_+$. Fix $x>0$ the approximation of \eqref{XT} will be given by the total mass $X^{N,x}$ of a population of individuals, each of which has mass $1/N$. The initial mass is $X_0^{N,x}= [Nx]/N$, and $X^{N,x}$ follows a Markovian jump dynamics :  from its current state $k/N$,  
 \begin{equation*}
X^{N,x} \  \mbox{jumps to} \ 
\left\{
    \begin{array}{ll}
   \frac{k+\ell-1}{N} \ \mbox{at rate} \  \psi_{\delta_N}^{\prime} (N) \nu_{N}(\ell) k,  \ \mbox{for all} \ \ell\ge2; &\\\\  \frac{k-1}{N}  \quad  \mbox{at rate} \  \psi_{\delta_N}^{\prime} (N) \nu_{N}(0) k.
    
&
           \end{array}
           \right.
\end{equation*}

In this process, each individual dies without descendant at rate
\[ \frac{\psi_{\delta_N}(N)}{N}=cN+\int_{\delta_N}^\infty z\mu(dz)-\frac{1}{N}\int_{\delta_N}^\infty (1-e^{-Nz})\mu(dz);\]
dies and leaves two descendants at rate
\[ cN+\frac{1}{N}\int_{\delta_N}^\infty \frac{(Nz)^2}{2}e^{-Nz}\mu(dz);\]
and finally dies and leaves $k$ descendants ($k\ge3$) at rate
\[ \frac{1}{N}\int_{\delta_N}^\infty\frac{(Nz)^k}{k!}e^{-Nz}\mu(dz).\]

Let $\mathcal{D}([0,\infty), \mathbb{R}_+)$ denote the space of functions from $[0,\infty)$ into $\mathbb{R}_+$ which are right continuous and have left limits at any $t>0$ (as usual such a function is called c\`adl\`ag). We shall always equip the space $\mathcal{D}([0,\infty),\mathbb{R}_+)$ with the Skorohod topology. The main limit proposition of this subsection is a consequence of Theorem 4.1 in \cite{DP}. 
\begin{proposition}\label{TH1}
 Suppose that Assumptions $(\bf H)$ is satisfied. Then, as $N\rightarrow +\infty$, $\{X_t^{N,x}, \ t\geqslant0 \}$ converges to $\{X_t^x, \ t\geqslant0 \}$ in distribution on $\mathcal{D}([0,\infty),\mathbb{R}_+)$, where $X^x$ is the unique solution of the SDE \eqref{XT}.  
\end{proposition}

\subsection{The approximate height process}

We shall now define $\{H^{N}(s), \ s\geq0\}$, the height process associated to the population process $\{X^{N,x}_t,\ t\ge0\}$. 

Before making precise the evolution of $H^N$, we need to define its local time $L_s^N(t)$, accumulated by $H^{N}$ at level $t$ up to time $s$ by 
\begin{equation}\label{TLint}
  L_s^N(t) = \lim_{\varepsilon\mapsto 0}\frac{1}{\varepsilon} \int_{0}^{s}\mathbf{1}_{\{t\leq H^N(r)< t+\varepsilon\}}dr.
\end{equation}
$L_s^N(t)$ equals $1/ N$ times the number of pairs of $t$-crossings of $H^{N}$  between times $0$ and $s$. In other words, $L_s^N(t)$ equals $(1/2)*(1 / N)$ times the number of visits at level $t$. Note that this process is neither right- nor left-continuous as a function of $s$.

We shall describe several Poisson processes. All of them will be mutually independent, even if we do not repeat it. Let first $\{{Q}_{s}^N, \ s\geq 0\}$ be a Poisson process with intensity $2\int_{\delta_N}^\infty(1-e^{-Nz}-Nze^{-Nz}) \mu(dz)$. This process will describe the ``arrival'' of mutiple births.  Let $\{{P}_{s}^{N}, \ s\geq 0\}$ and $\{{P}_{s}^{\ast,N}, \ s\geq 0\},$  be two mutually independent Poisson processes with respective intensities $2cN^2$ and\\ $2\int_{\delta_N}^{\infty}(e^{-Nz}-1+Nz)\mu(dz)$, which are globally independent of $\Pi$\footnote{Note that the above intensities are the rates of brith and death of the population process $X^{N,x}$, multiplied by $2N$. The slope of $H^N$ is $\pm2N$, which explains the factor $2N$.}.  Let us define ${P}_{s}^{N,-}= {P}_{s}^{N}+{P}_{s}^{\ast,N}$, $\forall$ $s\ge0$. Let $\{Z^N_i, \ i\geqslant1 \}$ be a sequence of i.i.d r.v.'s taking their values in the set 
$\{k/N,\ k\ge2\}$, which are independent of the Poisson processes, and whose law is precised as follows. 
\[ \P(Z^N_i=k/N)= \left(\int_{\delta_N}^\infty(1-e^{-Nz}-Nze^{-Nz}) \mu(dz)\right)^{-1}\int_{\delta_N}^\infty \frac{(Nz)^k}{k!}e^{-Nz}\mu(dz),\, k\ge2.\]

Let $\{V_{s}^{N}, \ s\geq 0\}$  be the c\`adl\`ag $\{-1,1\}$-valued process which  is such that, $s-$almost everywhere,  ${dH^{N}(s)}/{ds}=2NV_{s}^{N}.$  The $({\mathbb{R}}_{+} \times \{-1,1\})$-valued process  $\{(H^{N}(s), V_{s}^{N}), \ s\geq 0\}$ solves the SDE
\begin{equation}\label{DEFHNVN}
\begin{split}
H^{N}(s)=&2 N\int_{0}^{s}V_{r}^{N}dr,  \\
V_{s}^{N}=&1+ 2\int_{0}^{s}\mathbf{1}_{\{V_{r^{-}}^{N}=-1\}}dP_{r}^{N}+2\int_{0}^{s}\mathbf{1}_{\{V_{r^{-}}^{N}=-1\}}dQ_{r}^N
\\ &\quad -2\int_{0}^{s}\mathbf{1}_{\{V_{r^{-}}^{N}=+1\}}dP_{r}^{N,-}  + 2N \big(L_s^{N}(0)- L_{0^+}^{N}(0)\big)
\\ &\quad+ 2N \sum_{i>0, S^N_{i}\leq s}\Big\{L_{s}^{N}\left(H^{N}(S^N_{i})\right)-L_{S^N_{i}}^{N}\left(H^{N}(S^N_{i})\right) \Big\}\wedge\left(Z^N_i-\frac{1}{N}\right),
\end{split}
\end{equation}
where the $S^N_{i}$'s are the successive jump times of the process 
\[ \tilde{Q}^N_s=\int_0^s\mathbf{1}_{\{V_{r^{-}}^{N}=-1\}}dQ^N_r.\] 
For any $i>0$, $NZ^N_i-1$ denotes the number of reflections of $H^{N}$ above the level $H^{N}(S^N_{i})$ before the process $H^N$ may go below that level $H^{N}(S^N_{i})$.

\subsection{Taking the limit in the SDE for $H^N$}

We write the first line of \eqref {DEFHNVN} as 
\begin{equation*}
 cH^{N}(s)= 2cN\int_{0}^{s}\mathbf{1}_{\{V_{r}^{N}=+1\}}dr - 2cN\int_{0}^{s}\mathbf{1}_{\{V_{r}^{N}=-1\}}dr\,.
\end{equation*}
Summing it with the second identity in \eqref {DEFHNVN} divided by $2N$,
using the notations
\begin{align}\label{DIS}
\mathcal{M}_s^{1,N} &=   \frac{1}{N}  \int_{0}^{s}\mathbf{1}_{\{V_{r^{-}}^{N}=-1\}} \big(dP_r^{N} -2cN^2dr\big),\nonumber
\\\mathcal{M}_s^{2,N}& =   \frac{1}{N} \int_{0}^{s}\mathbf{1}_{\{V_{r^{-}}^{N}=+1\}}\big(d{P}_{r}^{N}-2cN^2dr\big),
\\\mathcal{M}_s^{\ast,N} &=  \frac{1}{N} \int_{0}^{s}\mathbf{1}_{\{V_{r^{-}}^{N}=+1\}}\big(d{P}_{r}^{\ast,N}
-2\int_{\delta_N}^{\infty}(e^{-Nz}-1+Nz)\mu(dz)dr\big),\nonumber
\end{align}
and the identity $a\wedge b=b-(b-a)^+$ for  $a,b>0$, we obtain
\begin{align}\label{2DFHNint}
 cH^N(s) &= Y^N(s)+L_s^{N}(0) -  \sum_{i>0, S^N_{i}\leq s} \bigg( Z^N_i-\frac{1}{N}- \left\{L_{s}^{N}\left(H^{N}(S^N_{i})\right)-L_{S^N_{i}}^{N}\left(H^{N}(S^N_{i})\right) \right\}\bigg)^+  ,
\end{align}
where
\begin{align}
Y^N(s)&=\mathcal{M}_s^{1,N}-\mathcal{M}_s^{2,N} +\mathcal{M}_s^{N}+ \epsilon^{N}(s),\label{defYN}\\
\label{eq:EpsiN}
\epsilon^{N}(s)
&= \frac{1}{2N} -  \frac{V_{s}^{N}}{2N}-{\mathcal{M}}_s^{\ast,N} - L_{0}^{N}(0)+  \int_{\delta_N}^{\infty}z\mu(dz) \left( s- 2\int_{0}^{s}\mathbf{1}_{\{V_{r}^{N}=+1\}}dr \right)\\
&\quad+\frac{2}{N}\int_{\delta_N}^\infty (1-e^{-Nz})\mu(dz)\int_0^s\mathbf{1}_{\{V_{r}^{N}=+1\}}dr,\nonumber\\
\mathcal{M}_s^N&=  \int_{0}^{s} Z^N_{\tilde{Q}^N_r}\, d\tilde{Q}^N_r- s\int_{\delta_N}^{\infty}z\mu(dz). \nonumber
\end{align}
In the following, we will need the equation \eqref{eq:VNs} below.  Writing $V^N_r$ as 
\[\mathbf{1}_{\{V_{r^{-}}^{N}=+1\}} - \mathbf{1}_{\{V_{r^{-}}^{N}=-1\}} \]
and using \eqref{DIS}, we deduce from \eqref{DEFHNVN}, 
\begin{equation}\label{eq:VNs}
\begin{split}
V^N_s =& 1+2N \left(\mathcal{M}_s^{1,N}-\mathcal{M}_s^{2,N} -{\mathcal{M}}_s^{\ast,N} +\mathcal{M}_s^{N} \right) - 4cN^2 \int_0^s V^N_r dr \\
&\quad+4\int_{\delta_N}^\infty (1-e^{-Nz})\mu(dz)\int_0^s\mathbf{1}_{\{V_{r}^{N}=+1\}}dr
\\&\quad+  2N \int_{\delta_N}^{\infty}z\mu(dz) \left( s- 2\int_{0}^{s}\mathbf{1}_{\{V_{r}^{N}=+1\}}dr \right)+  2N \big(L_s^{N}(0)- L_{0^+}^{N}(0)\big)\\
&\quad -2N \sum_{i>0, S^N_{i}\leq s} \bigg( Z^N_i-\frac{1}{N}- \left\{L_{s}^{N}\left(H^{N}(S^N_{i})\right)-L_{S^N_{i}}^{N}\left(H^{N}(S^N_{i})\right) \right\}\bigg)^+ .
\end{split}
\end{equation}

We first need an apriori estimate on the sequence of processes $H^N$.
\begin{proposition}\label{aeHN}
For any $s>0$, 
\[ \sup_{N\ge1}\E\left(\sup_{0\le r\le s}H^N(r)\right)<\infty\, .\]
\end{proposition}
\begin{proof} 
We first recall that by construction, $H^N(s)\ge0$, for all $s\ge0$, a.s. In this proof, we shall make use of several results to be proven below. The reader can easily verify that each of those results will be proven independently of the present result.

We first note that
\[ \int_{0}^{s}\mathbf{1}_{\{V_{r}^{N}=1\}}dr+\int_{0}^{s}\mathbf{1}_{\{V_{r}^{N}=-1\}}dr=s,\]
and from the identity $V^N_r=\mathbf{1}_{\{V_{r}^{N}=1\}}-\mathbf{1}_{\{V_{r}^{N}=-1\}}$ and the first line of \eqref{DEFHNVN}, we deduce that
\begin{equation*}
\int_{0}^{s}\mathbf{1}_{\{V_{r}^{N}=1\}}dr-\int_{0}^{s}\mathbf{1}_{\{V_{r}^{N}=-1\}}dr=(2N)^{-1}H^N(s)\,.
\end{equation*}
It follows from those two identities that
\begin{equation}\label{surN} 
 (2N)^{-1}H^N(s)=2\int_{0}^{s}\mathbf{1}_{\{V_{r}^{N}=1\}}dr-s\ge0\,.
 \end{equation}
We will prove in Lemma \ref{LN0INF} below that $L^N_s(0)=-\inf_{0\le r\le s}Y^N_r$.  From \eqref{2DFHNint}, \eqref{eq:EpsiN} and the last two identities, we deduce that
\begin{equation}\label{eq:aeHN}
\begin{split}
cH^N(s)&\le\sup_{0\le r\le s}\left\{|\mathcal{M}_r^{1,N}|+|\mathcal{M}_r^{2,N}|+|\mathcal{M}_r^{N}|+|{\mathcal{M}}_r^{\ast,N}|\right\}+N^{-1}
+L^N_0(0)\\&\quad+ \frac{2s}{N}\int_{\delta_N}^\infty\mu(dz)+(2N)^{-1}\int_{\delta_N}^\infty z\mu(dz)\sup_{0\le r\le s} H^N(r)\,.
\end{split}
\end{equation}
Since from assumption $(\bf A)$, $(2N)^{-1}\int_{\delta_N}^\infty z\mu(dz)\to0$ as $N\to\infty$, the result follows from the next four facts.
 Concerning $\mathcal{M}_s^{1,N}$, we deduce from the first line of \eqref{DIS} and Doob's $L^2$ inequalities for martingales that
 \[ \E\left(\sup_{0\le r\le s}|\mathcal{M}^{1,N}_r|\right)\le2\left(\E\left\{|\mathcal{M}^{1,N}_r|^2\right\}\right)^{1/2}\le2\sqrt{2c s}\,.\]
 The same holds concerning $\mathcal{M}_s^{2,N}$. As for $\mathcal{M}_r^N$, we consider \eqref{MNs} below. We first note that
 the expectation of the sup in $s$ of the absolute value of the sum of the last two terms is bounded by
 \[\frac{s}{N}\int_{\delta_N}^\infty\mu(dz)+(2N)^{-1}\int_{\delta_N}^\infty z\mu(dz)\ \E\left(\sup_{0\le r\le s}H^N(r)\right),\]
 which we can plug in \eqref{eq:aeHN} after we have replaced on the left $H^N(s)$ by $\sup_{0\le r\le s}H^N(r)$ and taken the expectation.
 We now consider the first term on the right of \eqref{MNs}, see \eqref{MNtilde}. We deduce from the Burkholder--David--Gundy inequlity for possibly discontinuous martingales (see e.g. Theorem IV.48 in \cite{PhP}) that there exists a constant $C>0$ such that
 (we will use a formula from Remark \ref{remarmuN} below)
 \begin{align*}
 \E\left(\sup_{0\le r\le s}|\tilde{\mathcal{M}}^{N}_r|\right)&\le C\,\E\left\{\left(\int_0^s\int_0^\infty z^2\Pi^N(dr,dz)\right)^{1/2}\right\}\\
 &\le C\,\E\left\{\left(\int_0^s\int_0^1 z^2\Pi^N(dr,dz)\right)^{1/2}\right\}+C\,\E\left\{\left(\int_0^s\int_1^\infty z^2\Pi^N(dr,dz)\right)^{1/2}\right\}
 \\
 & \le C\left(s \int_{\delta_N}^1 z^2\mu(dz)\right)^{1/2}+C\,s\int_1^\infty z\mu(dz),
 \end{align*}

 which is bounded, uniformly w.r.t. $N$. It remains to consider the term $|{\mathcal{M}}_r^{\ast,N}|$. It will be shown below in the proof of Lemma \ref{lemMet} that $\E\sup_{0\le r\le s}|{\mathcal{M}}_r^{\ast,N}|\to0$, as $N\to\infty$. This concludes the proof.
\end{proof}

The first two identities in the previous proof, combined with the just obtained result, clearly yield the following essential result.
\begin{lemma}\label{Ssur2int}
 As $N\longrightarrow \infty$,
\begin{equation*}
\int_{0}^{s}\mathbf{1}_{\{V_{r}^{N}=1\}}dr\longrightarrow \frac{s}{2}; {~~~}\int_{0}^{s}\mathbf{1}_{\{V_{r}^{N}=-1\}}dr\longrightarrow \frac{s}{2}
\end{equation*}
a.s., locally uniformly in $s$.
\end{lemma}

We shall need (recall \eqref{DIS})
\begin{lemma}\label{lemMet}
As $N\longrightarrow\infty$, $$\Big(\mathcal{M}_s^{\ast,N}, \ s\geq0\Big)\longrightarrow  0 \ \mbox{ in} \ \mbox{probability}, \ \mbox{locally} \ \mbox{uniformly} \ \mbox{in} \ \mbox{s}.$$ 
\end{lemma}
\begin{proof}
Since $\mathcal{M}_s^{\ast,N}$ is a purely discontinuous local martingale, we deduce from \eqref{DIS} that
\begin{equation*}
\!\left[\mathcal{M}^{\ast, N}\right]_s=\frac{1}{N^2}\! \int_{0}^{s}\!\mathbf{1}_{\{V_{r}^{N}=+1\}} \!dP_r^{\ast,N}  \quad \mbox{and} \quad 
\langle{\mathcal{M}^{\ast,N}\rangle}_{s}=\!\frac{2}{N^2}\int_{\delta_N}^{\infty}(e^{-Nz}-1+Nz)\mu(dz) \int_{0}^{s}\! \mathbf{1}_{\{V_{r}^{N}=+1\}}\!dr.
\end{equation*}
From the Cauchy--Schwartz and Doob's $L^2$-inequality for martingales, we deduce that
\begin{equation*}
\mathbb{E}\left(\sup_{0\leq r\leq r} \left|{\mathcal{M}}_r^{\ast,N}\right| \right) \leq \sqrt{\frac{2s}{N} \int_{\delta_N}^{\infty}z\mu(dz)},
\end{equation*}
which tends to $0$ as $N\to\infty$ from assumption {(\bf A)}.
\end{proof}

Recalling \eqref{2DFHNint}, we have the 
\begin{proposition}\label{313}
As $N\longrightarrow\infty,$ $\epsilon^{N}(s)\longrightarrow  0 \ \mbox{ in} \ \mbox{probability}, \ \mbox{locally} \ \mbox{uniformly} \ \mbox{in} \ \mbox{s}.$
\end{proposition}
\begin{proof}
From \eqref{TLint}, it is easily checked that $L_{0}^{N}(0)= 1/2N$.  $N^{-1}V_{s}^{N}\longrightarrow0$ a.s uniformly with respect to $s$.  However, from \eqref{surN}, we have that  
\begin{equation*}
 \left| \frac{s}{2}- \int_{0}^{s}\mathbf{1}_{\{V_{r}^{N}=+1\}}dr \right| = \frac{1}{4N} H^N(s).
\end{equation*}
Combining this with assumption {(\bf A)}  and Proposition \ref{aeHN}, we deduce that 
\begin{equation*}
 \left( \frac{s}{2}- \int_{0}^{s}\mathbf{1}_{\{V_{r}^{N}=+1\}}dr \right) \int_{\delta_N}^{\infty}z\mu(dz) \longrightarrow  0 \ \mbox{ in} \ \mbox{probability}, \ \mbox{locally} \ \mbox{uniformly} \ \mbox{in} \ \mbox{s}.
\end{equation*}
The result follows by combining these arguments with \eqref{eq:EpsiN} and Lemma \ref{lemMet}.
\end{proof}

Let $\Lambda^N=\sum_{j\ge1}\delta_{(T^N_j,\tilde{Z}^N_j)}$, where the $T^N_j$'s are the jump times of the Poisson process $Q^N$, and the $\tilde{Z}^N_j$'s are i.i.d., independent of the $T^N_j$'s, with the same law as the $Z^N_i$'s. We can couple the two point processes
$\Lambda^N$ and $\Pi^N=\sum_{i\ge1}\delta_{(S^N_i,Z^N_i)}$ in such a way that
\begin{equation*}
\Pi^N(ds,dz)=\mathbf{1}_{V^N_{s-}=-1}\Lambda^N(ds,dz).
\end{equation*}
It is not hard to see (exploiting e.g. Corollary VI.3.5 in Cinlar \cite{EC}) that $\Lambda^N$ is a Poisson Point Process with mean measure $2ds\mu_N(dz)$, where $\mu_N$ is a measure on $(0,+\infty)$ which is supported on the set $\{k/N,\, k\ge2\}$, and is specified by 
\[ \mu_N(\{k/N\})=\int_{\delta_N}^\infty\frac{(Nz)^k}{k!}e^{-Nz}\mu(dz)\,.\]
Let us establish
\begin{lemma}\label{conv-muN}
The sequence $\mu_N$ converges to $\mu$ as $N\to\infty$, in the sense of weak convergence of measures on $(0,+\infty)$.
\end{lemma}
\begin{proof}
It suffices to show that for any $f\in C(0,+\infty)$ with compact support, $\mu_N(f)\to\mu(f)$. But
\begin{align*}
\mu_N(f)=\int_{\delta_N}^\infty \sum_{k\ge2} f\left(\frac{k}{N}\right)\frac{(Nz)^k}{k!}e^{-Nz}\mu(dz)
\to \int_0^\infty f(z)\mu(dz),
\end{align*}
as $N\to\infty$. The pointwise convergence of the integrand follows from the fact that, for fixed $z$, if $\xi_N$ denotes a Poi$(Nz)$ r.v.,
since $f(0)=f(1/N)=0$, at least for $N$ large enough,
\[ f_N(z):=\sum_{k\ge2}f\left(\frac{k}{N}\right)\frac{(Nz)^k}{k!}e^{-Nz}=\E f\left(\frac{\xi_N}{N}\right)\to f(z),\]
as $N\to\infty$ from the law of large numbers. Suppose that supp$(f)\subset[a,+\infty)$. Lebesgue's dominated convergence theorem implies that
\[ \int_{a/2}^\infty f_N(z)\mu(dz)\to \int_{a/2}^\infty f(z)\mu(dz) .\]
It remains to show that $\int_0^{a/2}f_N(z)\mu(dz)\to0=\int_0^{a/2}f(z)\mu(dz)$. But for $z\le a/2$, 
\begin{align*}
f_N(z)&=\sum_{k\ge aN}f\left(\frac{k}{N}\right)\frac{(Nz)^k}{k!}e^{-Nz},\\
|f_N(z)|&\le \|f\|_\infty\P(\xi_N>aN)
\le \frac{4\|f\|_\infty}{a^2}\frac{z}{N},
\end{align*}
where we have used the fact that Var$(\xi_N/N)=z/N$. The result follows, since $N^{-1}\int_{\delta_N}^\infty z\mu(dz)\to0$, again from assumption 
$(\mathbf{A})$.
\end{proof}

We now deduce from Lemma \ref{Ssur2int} and Lemma \ref{conv-muN}
\begin{proposition}\label{pro:convPiN}
As $N\to\infty$, $\Pi^N\Rightarrow \Pi$, in the sense of weak convergence in distribution of random probability measures, where $\Pi$  is a Poisson Point Process with mean measure $ds\mu(dz)$.
\end{proposition}
\begin{proof}
In view of Lemma \ref{conv-muN}, all we need to show is that for any $z>0$ such that $\mu(\{z\})=0$, 
$\Pi^N(\cdot,(z,+\infty))\Rightarrow\Pi(\cdot,(z,+\infty))$. We first note that $\Lambda^N$ converges to a PPP $\Lambda$,
whose mean measure is twice that of $\Pi$. Next, since $\Pi^N$ is dominated by $\Lambda^N$, it is tight (see the criterion in Lemma 16.15 of \cite{OK}), hence it converges along a subsequence to some limiting measure $\widetilde{\Pi}$, which must be a simple point measure, by comparison with $\Lambda$. We shall not distinguish the subsequence from the original one, by an abuse of notation. Let $\HH^N_s$ (resp. $\HH_s$) denote the filtration generated by the process $H^N$ (resp. by its limit in law $H$). We have that
\begin{align*}
 \Pi^N((0,s],(z,\infty))-2\mu_N(z,\infty)\int_0^s\mathbf{1}_{\{V_{r}^{N}=-1\}}dr\ \ \text{is an }\HH^N_s\  \text{martingale}.
 \end{align*}
 This implies that for any $n\ge1$, $0<s_1<\cdots<s_n=s<s'$, any bounded $\Phi\in C(\R^n;\R)$
 \[ \E\left[\Phi(H^N_{s_1},\ldots,H^N_{s_n})\Pi^N((s,s'],(z,+\infty))\right]= 2\mu_N(z,\infty)\E\left[\Phi(H^N_{s_1},\ldots,H^N_{s_n})\int_s^{s'}\mathbf{1}_{\{V_{r}^{N}=-1\}}dr\right]\,.\]
   Taking the limit in this last identity yields that
  \[ \E\left[\Phi(H_{s_1},\ldots,H_{s_n})\widetilde{\Pi}((s,s'],(z,+\infty))\right]= \mu(z,\infty)\E\left[\Phi(H_{s_1},\ldots,H_{s_n})(s'-s)\right]\,.\]
  This being true for all $n\ge1$, all $0<s_1<\cdots<s_n=s<s'$ and all bounded $\Phi\in C(\R^n;\R)$, we have that
  the simple point process $\widetilde{\Pi}((0,s],(z,\infty))$ is such that $\widetilde{\Pi}((0,s],(z,\infty))-s\mu(z,\infty)$ is a martingale. This shows that it is a Poisson process with intensity $\mu(z,\infty)$. Hence 
 $\widetilde{\Pi}$ is a PPP with mean measure $ds\mu(dz)$, so it has the same law as $\Pi$.
  \end{proof}
\begin{remark}\label{remarmuN}
From the definition of $\mu_N$, we have
\[ \int_{0}^{\infty}\mu_N(dz)= \int_{\delta_N}^{\infty}\mu(dz), \quad \mbox{and} \quad  \int_{0}^{\infty}z\mu_N(dz)= \int_{\delta_N}^{\infty}z\mu(dz) .\]
\end{remark}
Hence, we can rewrite $\mathcal{M}^N$, which appeared in \eqref{defYN}, in the following form  
\begin{align}
\mathcal{M}_s^N&= \int_{0}^{s} Z^N_{\tilde{Q}^N_r}\, d\tilde{Q}^N_r- s\int_{\delta_N}^{\infty}z\mu(dz)\nonumber \\
&= \mathcal{\tilde{M}}_s^N   -  \int_{\delta_N}^{\infty}z\mu(dz) \left( s- 2\int_{0}^{s}\mathbf{1}_{\{V_{r}^{N}=-1\}}dr \right),\label{MNs}\\
\mathcal{\tilde{M}}_s^N&=\int_{0}^{s} \int_{0}^{\infty} z \overline{\Pi}^N(dr,dz)\, . \label{MNtilde}
\end{align}
Using the same arguments as in lemma \ref{313}, we prove that the second term tends to $0$ as $N\to\infty$, locally uniformly in $s$.

 We now deduce from the above. 
\begin{corollary}\label{MNversM}
As $N\longrightarrow \infty$, $\mathcal{M}^N \Rightarrow \mathcal{M}$ in $\mathcal{D}([0,\infty))$, where 
\begin{equation*}
 \mathcal{M}_s= \int_{0}^{s} \int_{0}^{\infty}z \overline{\Pi}(dr,dz).
\end{equation*}
\end{corollary}
\begin{proof}
We split $\tilde{\mathcal{M}}^N_s$ into two terms.
For any $\delta>0$ such that $\mu({\delta})=0$, one can deduce from Proposition \ref{pro:convPiN} that
as $N\to\infty$,
\[ \int_0^\cdot \int_{\delta}^\infty z\overline{\Pi}^N(dr,dz) \Rightarrow\int_0^\cdot\int_{\delta}^\infty z\overline{\Pi}(dr,dz)
\ \text{in }\mathcal{D}([0,\infty)).\]
On the other hand,
\begin{align*}
\E\left[\sup_{0\le s\le \bar{s}}\left(\int_0^s\int_{0}^{\delta} z\overline{\Pi}^N(dr,dz)\right)^2\right]&\le8 \ \E\int_0^{ \bar{s}}\mathbf{1}_{V^N_{r}=-1}dr\int_0^\delta z^2\mu_N(dz)\\
&\to4 \bar{s}\int_0^\delta z^2\mu(dz),
\end{align*}
while
\[\E\left[\sup_{0\le s\le \bar{s}}\left(\int_0^s\int_{0}^{\delta} z\overline{\Pi}(dr,dz)\right)^2\right]\le 4 \bar{s}\int_0^\delta z^2\mu(dz)\,.\]
Since $\int_0^\delta z^2\mu(dz)$ can be made arbitrarily small by choosing $\delta>0$ small enough, the result follows from the above statements by standard arguments.
\end{proof}


The following result is essentially very classical
\begin{lemma}\label{M1NM2N}
As $N\longrightarrow\infty$, $$\Big(\mathcal{M}_s^{1,N},\mathcal{M}_s^{2,N}, \ s\geq0\Big)\Longrightarrow \left( \sqrt{c}
B_{s}^{1},  \sqrt{c} B_{s}^{2}, \ s\geq0\right) \ in \ {(\mathcal{D}([0,\infty)))}^{2},$$ where $B_{s}^{1}$ and $B_{s}^{2}$ are two mutually independent standard Brownian motions.
\end{lemma}
\begin{proof}
The precise result here, with the identification of the constants is Proposition 5.3 in \cite{LePW}, see also Proposition 4.23 in \cite{DPS}.
\end{proof}

Let us define 
\begin{equation}\label{BNS}
B^N(s)=  \mathcal{M}_s^{1,N}-\mathcal{M}_s^{2,N}.
\end{equation}
We deduce readily 
\begin{corollary}\label{V12}
As $N\longrightarrow\infty$, $$\Big(B^{N}(s), \ s\geq0\Big)\Longrightarrow \left( \sqrt{2c} B(s), \ s\geq0\right) \ in \ {(\mathcal{D}([0,\infty)))},$$ 
where $B$ is a standard Brownian motion.
\end{corollary}
Let us rewrite \eqref{2DFHNint} in the form
\begin{align}\label{2DFHNintnew}
 cH^{N}(s) =&Y^N(s) + L_s^{N}(0) -\mathcal{R}^{N}(s), 
\end{align}
with
\begin{equation}\label{VC}
Y^N(s)=\epsilon^{N}(s) + B^N(s) + \mathcal{M}_s^N, 
\end{equation}
and
\begin{equation*}\label{eq:RN}
 \mathcal{R}^{N}(s) =  \sum_{i>0, S^N_{i}\leq s} \bigg(Z_i-\frac{1}{N}-  \left\{L_{s}^{N}\left(H^{N}(S^N_{i})\right)-L_{S_{i}^N}^{N}\left(H^{N}(S^N_{i})\right) \right\}\bigg)^+.
\end{equation*}

We have proved so far that $B^N\Rightarrow \sqrt{\frac{2}{c}}  B$ and $\epsilon^{N}+\mathcal{M}^N\Rightarrow \mathcal{M}$, as $N\to\infty$. Then taking the weak limit in \eqref{VC}, we have in fact the  
\begin{corollary}\label{V1}
 As $N\longrightarrow \infty$, $Y^N \Longrightarrow Y $ in $\mathcal{D}([0,\infty))$, where 
\begin{equation*}
Y(s)=\sqrt{2c}B(s) + \int_{0}^{s} \int_{0}^{\infty}z\overline{\Pi}(dr,dz).
\end{equation*}
\end{corollary}
\begin{proof}
Also we have taken separately the limit in $B^N$ and in $\mathcal{M}^N$, it is clear that, at least along a subsequence,
the pair converges, and the limits are independent, since one is Brownian motion, and the other is a	 Poisson integral, both being martingales w.r.t. the same filtration.
\end{proof}

Note that since the jump times do not move only the sequence, in fact $Y^N\Rightarrow Y$ for the topology of locally uniform convergence.
We shall need the following
\begin{lemma}\label{LN0INF}
For any $N\ge1, s>0$, 
\[ L^N_s(0)=-\inf_{0\le r\le s}Y^N(r).\]
\end{lemma}
 \begin{proof}
Recall \eqref{2DFHNintnew}. We first note that, since $L^N_\cdot(0)$ increases only when $H^N(s)=0$, $L^N_s(0)=L^N_{r_s}(0)$, where $r_s=\sup\{r\le s ; H^N(r)=0\}$. We also have $ \mathcal{R}^{N}(r_s)=0$. Consequently
\[  L^N_s(0)=L^N_{r_s}(0)=-Y^N(r_s)\le -\inf_{0\le r\le s}Y^N(r).\]
To establish the converse inequality, let $0\le u_s\le s$ be such that $Y^N(u_s)=\inf_{0\le r\le s}Y^N(r)$. Since  $Y^N(u_s)+L^N_{u_s}(0)=cH^N(u_s)+ \mathcal{R}^{N}(u_s)\ge0$, 
\[ L_s^N(0)\ge L^N_{u_s}(0)\ge -Y^N(u_s)=-\inf_{0\le r\le s}Y^N(r)\,.\]
 \end{proof}

We have similarly 
 \begin{lemma}\label{LNLNINF}
 For any $N\ge1$, $i\ge1$ such that $S^N_i\le {s}$,   
 \[L_{s}^{N}\left(H^{N}(S^N_{i})\right)-L_{S^N_{i}}^{N}\left(H^{N}(S^N_{i})\right)= -\inf_{S^N_{i}\le r\le s}\left( Y^N(r)-Y^N(S^N_{i})\right), \]
 for $s\ge S^N_i$ such that $L_{s}^{N}\left(H^{N}(S^N_{i})\right)-L_{S^N_{i}}^{N}\left(H^{N}(S^N_{i})\right)\le 
Z^N_i -1/N$.
 \end{lemma}
 \begin{proof}
 The argument is the same as in the previous Lemma. On the considered time interval, $H^N(s)$ is reflected above the level $H^N(S_i)$, instead of being reflected above $0$.
  \end{proof}

From Lemma \ref{LN0INF}, we can rewrite \eqref{2DFHNintnew} in the form 
\begin{align}
 cH^{N}(s) &=Y^N(s) -\inf_{0\le r\le s}Y^N(r) -\mathcal{R}^{N}(s),\ \text{where} \label{eq:HN}\\
 \mathcal{R}^{N}(s) &=  \sum_{i>0, S^N_{i}\leq s} \left( Z^N_i-\frac{1}{N} +\inf_{S^N_{i}\le r\le s}\left( Y^N(r)-Y^N(S^N_{i})\right) \right)^+.
 \label{formulaRN}
\end{align}

We also rewrite \eqref{CHS1} as 
\begin{equation}\label{eq:H}
\begin{split}
c H(s)&= Y(s) -\inf_{0\le r\le s}Y(r)-\mathcal{R}(s) ,\ \text{where} \\
  \mathcal{R} (s)&=\int_0^s\int_0^\infty\left(z + \inf_{r\leqslant u \leqslant s}Y(u)-Y(r) \right)^+\Pi(dr, dz).
\end{split}
\end{equation}
We next prove
\begin{proposition}\label{convHN}
As $N\to\infty$, $H^N\Rightarrow H$ in $\mathcal{C}(\R_+)$, where $H$ is given by \eqref{CHS1}, in terms of $Y$, the limit given in 
Corollary \ref{V1}.
\end{proposition}
\begin{proof}
It suffices to show that if $Y^N(s)\to Y(s)$ a.s., locally uniformly in $s$, then $H^N(s)\to H(s)$ in probability, locally uniformly in $s$.
We will show that from any subsequence, we can extract a further subsequence which converges a.s., locally uniformly in $s$.
We will follow closely the proofs of Proposition 3.13 and Corollary 3.15 in \cite{LiPW}. First of all, the same argument as that of Proposition 3.13
in \cite{LiPW} yields that $H^N(s)\to H(s)$ in probability, for any $s\ge0$. We fix $\bar{s}>0$ arbitrary, and let $\mathbf{D}_{\bar{s}}$ be a countable dense subset of $[0,\bar{s}]$.
Along a subsequence, still denoted as $H^N$ by an abuse of notation, $H^N(s)\to H(s)$, for any $s\in\mathbf{D}_{\bar{s}}$. We first note that, as in \cite{LiPW}, 
we can rewrite \eqref{CHS1} as follows. Let, for any $0\le r<s$, $\overline{Y}^s_r:=\inf_{r\le u\le s}Y(u)$. We have
\[ cH(s)=Y(s)-\overline{Y}^s_0-\sum_{0\le r\le s}\Delta \overline{Y}^s_r,\]
where $\Delta\overline{Y}^s_r$ denotes the jump at $r$ of the increasing function $r\mapsto\overline{Y}^s_r$.

Since the process $Y$ is c\`adl\`ag, with only positive jumps, 
\[ \Phi_Y(h):=\sup_{0\le r<s\le r+h\le\bar{s}} (Y(s)-Y(r))_-\] 
is a.s. a continuous function of $h$ on $[0,1]$, such that $\Phi_Y(0)=0$. Now
\begin{align*}
c(H(s+h)-H(s))&=Y(s+h)-Y(s)-\overline{Y}^{s+h}_0+\overline{Y}^s_0-\sum_{0\le r\le s}(\overline{Y}^{s+h}_r-\overline{Y}^s_r)
-\sum_{s<r\le s+h}\overline{Y}^{s+h}_r\\ 
&\ge Y(s+h)-Y(s)-\sum_{s<r\le s+h}\overline{Y}^s_r
\end{align*}
But since $Y(s+h)-\sum_{s<r\le s+h}\overline{Y}^{s+h}_r\ge\inf_{s<r\le s+h}Y(r)$, we have
\begin{align*} 
c(H(s+h)-H(s))&\ge \inf_{s<r\le s+h}Y(r)-Y(s),\ \text{hence}\\
(H(s+h)-H(s))_-\le \Phi_Y(h)\,.
\end{align*}
It follows from \eqref{eq:H} that the same formula relates $H^N$ with $Y^N$, and we also have
\[ (H^N(s+h)-H^N(s))_-\le \Phi_{Y^N}(h).\]
However, it is not true that $\Phi_{Y^N}(h)\to0$, as $h\to0$, since $Y^N$ has negative jumps of size $-1/N$. So $\limsup_{h\to0}\Phi_{Y^N}(h)\le1/N$. 
Moreover, since $Y^N(s)\to Y(s)$ uniformy on $[0,\bar{s}]$, we have that \\ $\limsup_{h\to0}\sup_{N\ge N_0}\Phi_{Y^N}(h)\le1/N_0$, for all $N_0\ge1$.
Now the fact that $H^N(s)\to H(s)$ uniformly on $\mathbf{D}_{\bar{s}}$ (hence also on $[0,\bar{s}]$) follows from the next Lemma. The result follows.
\end{proof}

It remains to establish
\begin{lemma}
Consider a sequence $\{g_N,\ N\ge1\}$ of c\`al\`ag functions from $\R_+$ into $\R$, $g$ a continuous function from $\R_+$ into $\R$, and $\bar{s}>0$, which are such that $g_N(s)\to g(s)$, for all 
$s\in\mathbf{D}_{\bar{s}}$. Assume moreover that
\[ \limsup_{h\to0}\sup_{N\ge N_0}(g_N(s+h)-g_N(s))_-\le1/N_0\,.\]
Then, as $N\to\infty$,
\[ \sup_{s\in\mathbf{D}_{\bar{s}}}|g_N(s)-g(s)|\to0.\]
\end{lemma}
The proof of this Lemma, which can be viewed as an extension of the second Dini theorem,  is identical to that of Lemma 3.16 in \cite{LiPW}, even if its statement is slightly different, so we do not reproduce it.

We now have

\begin{proposition}\label{TRIPLET}
As $N\longrightarrow \infty$, $(H^N,Y^N, \mathcal{R}^{N}) \Rightarrow (H,Y, \mathcal{R})$ in $\mathcal{C}([0,+\infty))\times {(\mathcal{D}([0,+\infty))}^{2}$.
\end{proposition}
\begin{proof} We can rewrite \eqref{eq:HN}, in the form 
\begin{equation*}
\mathcal{R}^{N}(s) =Y^N(s) -\inf_{0\le r\le s}Y^N(r) -cH^{N}(s).
\end{equation*}
It follows from Proposition \ref{convHN} that the sequence $\{H^{N},\ N\ge1\}$ is tight in $\mathcal{C}(\R_+)$. However, from Corollary \ref{V1}, we deduce that the sequences $\{Y^{N}(s),\ s\ge0\}_{N\ge1}$ and $\{\inf_{0\le r\le s} Y^{N}(r),\  s\ge0\}_{N\ge1}$ are tight in $\mathcal{D}([0,+\infty)$. Moreover, the limit of the sequence\\ $\{\inf_{0\le r\le s} Y^{N}(r),\  s\ge0\}_{N\ge1}$ is a.s. continuous. Hence the tightness of the sequence $\{\mathcal{R}^{N},\ N\ge1\}$ follows from Proposition \ref{TenD2} below. Now since $(H^N,Y^N, \mathcal{R}^{N})$ is tight, along an appropriate subsequence (which we do not distinguish from the original sequence),
\begin{equation*}
(H^N,Y^N, \mathcal{R}^{N}) \Rightarrow (H,Y, \mathcal{R}).
\end{equation*}
Moreover, from \eqref{eq:H} and the fact that the law of $Y$ is uniquely specified, we deduce that the limit is unique, which implies that the whole sequence converges.
\end{proof}

We shall need below the 
\begin{lemma}\label{ParZeng0}
For any $\delta>0$,  as $N\to\infty$,
\[ \int_0^\cdot \int_{0}^\delta  \left( z-\frac{1}{N} +\inf_{r\le u\le \cdot } Y^N(u)-Y^N(r) \right)^+ \Pi^N(dr,dz) \Rightarrow\int_0^\cdot\int_{0}^\delta \left(z + \inf_{r\leqslant u \leqslant \cdot }Y(u)-Y(r) \right)^+\Pi(dr, dz)
\]
in $\mathcal{D}([0,\infty))$.
\end{lemma}
\begin{proof}
Let us decompose $cH^N(s)$ and $cH(s)$ in the form 
\begin{align*}
cH^N(s)=\mathcal{B}^N_s+\mathcal{P}^N_s-\mathcal{R}_s^N,\quad
cH(s)= \mathcal{B}_s+\mathcal{P}_s-\mathcal{R}_s,
\end{align*}
where $\mathcal{R}_s^N$ (resp. $\mathcal{R}_s$) is defined by the second line of \eqref{eq:HN} (resp. of \eqref{eq:H}), 
\[ \mathcal{P}^N_s=\int_0^s\int_0^\infty z\overline{\Pi}^N(dr,dz),\quad \mathcal{P}_s=\int_0^s\int_0^\infty z\overline{\Pi}(dr,dz),\]
and $\mathcal{B}^N_s$ (resp. $\mathcal{B}_s$) is the remainer of $cH^N(s)$ (resp. of $cH(s)$). First we not that  $\mathcal{B}^N_\cdot\Rightarrow
\mathcal{B}_\cdot$ for the topology of uniform convergence (the limit is continuous). Next we introduce the decompositions
\begin{align*}
\mathcal{P}^N_{\delta,-}(s)&=\int_0^s\int_0^\delta z\overline{\Pi}^N(dr,dz),\quad \mathcal{P}^N_{\delta,+}(s)=\int_0^s\int_\delta^\infty z\overline{\Pi}^N(dr,dz),\\
\mathcal{P}_{\delta,-}(s)&=\int_0^s\int_0^\delta z\overline{\Pi}(dr,dz),\quad \mathcal{P}_{\delta,+}(s)=\int_0^s\int_\delta^\infty z\overline{\Pi}(dr,dz),\\
\mathcal{R}_{\delta,-}^N(s)&=\int_0^s\int_0^\delta\left(z-\frac{1}{N}+\inf_{r\le u\le s}Y^N(u)-Y^N(r)\right)^+\Pi^N(dr,dz),\\
\mathcal{R}_{\delta,+}^N(s)&=\int_0^s\int_\delta^\infty\left(z-\frac{1}{N}+\inf_{r\le u\le s}Y^N(u)-Y^N(r)\right)^+\Pi^N(dr,dz),\\
\mathcal{R}_{\delta,-}(s)&=\int_0^s\int_0^\delta\left(z+\inf_{r\le u\le s}Y(u)-Y(r)\right)\Pi(dr,dz),\\
\mathcal{R}_{\delta,+}(s)&=\int_0^s\int_\delta^\infty\left(z+\inf_{r\le u\le s}Y(u)-Y(r)\right)\Pi(dr,dz)\,.
\end{align*}
Let $\mathcal{C}^N(s):=\mathcal{P}^N_{\delta,+}(s)-\mathcal{R}_{\delta,+}^N(s)$, $\mathcal{C}(s)=\mathcal{P}_{\delta,+}(s)-\mathcal{R}_{\delta,+}(s)$.
We claim that  $\mathcal{C}^N(s)$ is tight (see Lemma \ref{DrPa} below), and that its limit is $\mathcal{C}(s)$ and it is continuous. So that the convergence is (locally) uniform in $s$.
Finally 
\[ \mathcal{R}_{\delta,-}^N(s)=-cH^N(s)+\mathcal{B}^N_s+\mathcal{C}^N(s)+\mathcal{P}_{\delta,-}(s)\,.\]
The sum of the first three terms on the right is tight and converges locally uniformly in $s$ towards its continuous limit $-cH(s)+ \mathcal{B}_s+\mathcal{C}(s)$, while the last term can be shown (so to speak ``by hand'') to be tight in $\mathcal{D}(\R_+)$.  Hence the right hand side is tight in $\mathcal{D}(\R_+)$, and so is the left hand side. Taking the weak limit in the last identity, we obtain that the limit of $ \mathcal{R}_{\delta,-}^N(s)$ is $ \mathcal{R}_{\delta,-}(s)$, which is our Lemma.
\end{proof}

We want to check the tightness of the sequence $\{\mathcal{C}^N, N\ge1\}$ using the Aldous criterion (see section 16, page 176 in \cite{Bill}). Let $\tau$ be a stopping time with value in $[0,s]$ and let $\epsilon>0$  be a real which will eventually go to zero.
\begin{lemma}\label{DrPa}
The sequence $\{\mathcal{C}^N, N\ge 1\}$ is tight in $\mathcal{D}(\R_+)$.
\end{lemma}
\begin{proof}
Recall the notations used in the previous proof. We have 
\begin{align*}
\mathcal{C}^N(s)&= \mathcal{C}_1^N(s) -  \mathcal{C}_2^N(s), \quad \mbox{where}
\\ \mathcal{C}_1^N(s)&= \int_0^s\int_\delta^\infty  \Upsilon_N(s,r,z) {\Pi}^N(dr,dz), \quad    \mathcal{C}_2^N(s)=2\int_0^{s}\mathbf{1}_{V^N_{r}=-1}dr \int_{\delta}^{\infty}    z\, \mu_N(dz),  \quad \mbox{with}
\\  \Upsilon_N(s,r,z) &=  \frac{1}{N}+ \left(Y^N(r)-\inf_{r\le u\le s}Y^N(u)\right) \wedge  \left( z-\frac{1}{N}  \right). 
\end{align*}
The tightness in $\mathcal{C}(\R_+)$ of the sequence $\mathcal{C}_2^N$ is rather clear. We have the following a priori estimates
\begin{align*}
&0\le \Upsilon_N(s,r,z) \le z, \quad \mbox{and if} \ r\le s <s',\\
&0\le \Upsilon_N(s',r,z) -  \Upsilon_N(s,r,z)\le Y^N(s)-\inf_{s\le u\le s'}Y^N(u).
\end{align*}
We now verify Aldous's criterion.
The first condition (see (16.22) in \cite{Bill}) follows easily from the inequality
\[\mathcal{C}_1^N(s) \le  \int_0^s\int_\delta^\infty  z\, {\Pi}^N(dr,dz). \]
We next want to establish the second condition (see (16.23) in \cite{Bill}), which will follow from the fact that for all $\eta>0$,
\[\lim_{\epsilon\rightarrow 0} \limsup_{N\rightarrow+\infty} \P \left(  \left|\mathcal{C}_1^N(\tau +\epsilon)-\mathcal{C}_1^N(\tau)  \right| > \eta \right)=0. \]
In order to verify this condition, we first note that
\begin{align*}
0\le \mathcal{C}_1^N(\tau +\epsilon)-\mathcal{C}_1^N(\tau)= & \int_{\tau}^{\tau+\epsilon}\int_\delta^\infty  \Upsilon_N(\tau+\epsilon,r,z)  {\Pi}^N(dr,dz)\\&+   \int_{0}^{\tau}\int_\delta^\infty \left[  \Upsilon_N(\tau+\epsilon,r,z) -  \Upsilon_N(\tau,r,z)\right]  {\Pi}^N(dr,dz)\\
\le&  \int_{\tau}^{\tau+\epsilon}\int_\delta^\infty z\, {\Pi}^N(dr,dz) + \left(Y^N(\tau)-\inf_{\tau \le u\le \tau +\epsilon}Y^N(u) \right) \Pi^N([0,\tau]\times [\delta,+\infty)).
\end{align*}
Now using the Portmanteau theorem, Corollary \ref{V1} and Markov's inequality, we deduce that
 \begin{align*}
 \limsup_{N\rightarrow+\infty} \P \left(  \left|\mathcal{C}_1^N(\tau +\epsilon)-\mathcal{C}_1^N(\tau)  \right| > \eta \right) \le & \frac{4}{\eta} \epsilon \int_{\delta}^\infty z\mu(dz)\\&+  \P \left(\left(Y(\tau)-\inf_{\tau\le u\le \tau +\epsilon}Y(u) \right) \Pi([0,\tau]\times [\delta,+\infty)) > \frac{\eta}{2} \right).
 \end{align*}
 However, using the strong Markov property of $Y$, we obtain
 \[\E \left\{\left(Y(\tau)-\inf_{\tau\le u\le \tau +\epsilon}Y(u) \right) \Pi([0,\tau]\times [\delta,+\infty)) \right\} =  \E\left(-\inf_{0\le u\le \epsilon}Y(u) \right) \E \left\{\Pi([0,\tau]\times [\delta,+\infty)) \right\}.\]
Combining this with the previous inequality and Markov's inequality, it follows that 
 \begin{align*}
 \limsup_{N\rightarrow+\infty} \P \left(  \left|\mathcal{C}_1^N(\tau +\epsilon)-\mathcal{C}_1^N(\tau)  \right| > \eta \right) \le & \frac{4}{\eta} \epsilon \int_{\delta}^\infty z\mu(dz)\\&+ \frac{2}{\eta} \E\left(-\inf_{0\le u\le \epsilon}Y(u) \right) \E \left\{\Pi([0,\tau]\times [\delta,+\infty)) \right\}
 \\&\xrightarrow[\epsilon \rightarrow 0] \ 0\,.
 \end{align*}
 thanks to the monotone convergence theorem. 
\end{proof}

\subsection{Convergence of the local time of the approximate height process}
The aim of this subsection is to pass to the limit in the process $\{L_s^N(H^N(s)), N\ge1\}$.  The proof of Theorem \ref{le:unifconvL} is carried out in two major parts. Recall \eqref{TLint} and Proposition \ref{Tanaka}. The first part (Proposition \ref{LNst-sfix}) provides the tightness of the sequence $\left\{ L_s^N(t),  t\ge 0 \right\}_{N\ge 1}$, for each $s\ge 0$ fixed, the second part (Lemma \ref{le:unifcontL}) establishes that the mapping $s\mapsto L_s(t)$ is continuous, uniformly in $t$.

We shall need below the
\begin{lemma}\label{kabe}
For any $s\ge0$, as $N\longrightarrow +\infty$,
\begin{equation*}
  \int_0^s\mathbf{1}_{\{H^N(r)>t\}} \mathbf{1}_{\{V_{r}^{N}=-1\}}dr\Rightarrow \frac{1}{2} \int_0^s  \mathbf{1}_{\{H(r)>t\}} dr.
\end{equation*}
\end{lemma}
 \begin{proof}
 Imitating again the proof of Lemma \ref{Ssur2int}, we have 
 \begin{align*}
 2\int_0^s \mathbf{1}_{\{H^N(r)>t\}}  \mathbf{1}_{\{V_{r}^{N}=-1\}}dr -\int_0^s\mathbf{1}_{\{H^N(r)>t\}}  dr
 &= \frac{1}{2N}\int_0^s\mathbf{1}_{\{H^N(r)>t\}} \frac{dH^N(r)}{dr} dr\\
 &=\frac{1}{2N}(H^N(s)-t)^+\, ,
\end{align*}
which clearly tends to $0$ in probability, as $N\to\infty$. It thus remains to take the limit in the sequence
$\{ \int_0^s\mathbf{1}_{\{H^N(r)>t\}} dr\}_{N\ge1}$.
For any $\delta >0$, we set $$\Gamma_\delta= \{r \in (0,s) ; |H(r)-t| \le \delta \}.$$  It follows from the properties of the process $H$ that Leb$(\Gamma_\delta)\to0$ a.s. as $\delta\to0$, where Leb$(A)$ denotes the Lebesgue measure of the set $A$.
We have 
\begin{align*}
\int_0^s \mathbf{1}_{\{H^N(r)>t\}}  dr &=   \int_{\Gamma_\delta}  \mathbf{1}_{\{H^N(r)>t\}}   dr  +  \int_{(0,s) \backslash \Gamma_\delta}  \mathbf{1}_{\{H^N(r)>t\}}dr, \\ 
\int_0^s  \mathbf{1}_{\{H(r)>t\}} dr&= \int_{\Gamma_\delta}  \mathbf{1}_{\{H(r)>t\}}  dr+ \int_{(0,s) \backslash \Gamma_\delta}  \mathbf{1}_{\{H(r)>t\}}  dr .
\end{align*}
 The two first terms on the right hand sides are dominated by Leb$(\Gamma_\delta)$. But for $r\in(0,s) \backslash \Gamma_\delta$,  
$\mathbf{1}_{\{H(r)>t\}} =g_\delta(H(r)-t)$, where $g_\delta$ is the continuous function from $\R$ into $[0,1]$ given as
\[ g_\delta(x)=1\wedge[\delta^{-1}(x+\delta/2)^+].\]
We clearly have
\[ \int_{(0,s) \backslash \Gamma_\delta}  \mathbf{1}_{\{H(r)>t\}}  dr=\int_{(0,s) \backslash \Gamma_\delta}  g_\delta(H(r)-t)dr,\
\text{and }\int_{(0,s) \backslash \Gamma_\delta}  g_\delta(H^N(r)-t)dr\Rightarrow\int_{(0,s) \backslash \Gamma_\delta}  g_\delta(H(r)-t)dr . \]
But since $\mathbf{1}_{\{H^N(r)>t\}}$ and $g_\delta(H^N(r)-t)$ differ only when $|H^N(r)-t|\le\delta/2$, if $h_\delta$ is a continuous function from $\R$ into $[0,1]$  which equals $1$ on $[-\delta/2,\delta/2]$, and $0$ outside $[-2\delta/3,2\delta/3]$, and
$k_\delta$ a function from $\R$ into $[0,1]$  which equals $0$ on $[-2\delta/3,2\delta/3]$ and $1$ outside $[-\delta,\delta]$,
\begin{align*}
\left|\int_{(0,s) \backslash \Gamma_\delta}  \mathbf{1}_{\{H^N(r)>t\}}dr-\int_{(0,s) \backslash \Gamma_\delta}  g_\delta(H^N(r)-t)dr\right|
\le\int_0^sh_\delta(H^N(r)-t)k_\delta(H(r)-t)dr,
\end{align*}
which tends to $0$ in probability.
 
\end{proof}

We shall also need below the
\begin{lemma}\label{in54}
For any $s>0$, $t>0$, the following identities hold a.s.
\begin{align*}
(H^N(s)-t)^+ &= 2N \int_0^s V_r^N \mathbf{1}_{\{ H^N(r) >t\}}dr\\
V_s^N \mathbf{1}_{\{ H^N(s) >t\}} &= 2N L_s^N(t) +  \int_0^s \mathbf{1}_{\{ H^N(r) >t\}}dV_r^N.
\end{align*}
 \end{lemma}
 \begin{proof}
 The proof is essentially the same as that of lemma 5.4 in \cite{LePW}.
 \end{proof}

 A first preparation for the proof of Theorem \ref{le:unifconvL} below is
 \begin{proposition}\label{LNst-sfix}
 For each $s\ge 0$ fixed, $\left\{ L_s^N(t),  t\ge 0 \right\}_{N\ge 1}$ is tight in $\mathcal{D}([0,+\infty))$.
 \end{proposition}
\begin{proof} 
Writing the second line of Lemma \ref{in54} as
\[  L_s^N(t)= \frac{1}{2N} V_s^N \mathbf{1}_{\{ H^N(s) >t\}} - \frac{1}{2N}  \int_0^s \mathbf{1}_{\{ H^N(r) >t\}}dV_r^N,\]
and using \eqref{eq:VNs}, \eqref{BNS}, \eqref{MNs}, Lemma \ref{LNLNINF} and the first line of Lemma \ref{in54}, we deduce that for any $t \ge 0$, a. s. 
 \begin{equation}\label{2DFHNintRFT1}
\begin{split}
 L_s^N(t)=& A_s^N(t) -  \int_0^s \mathbf{1}_{\{H^N(r)>t\}} dB^N(r) +  \int_0^s \mathbf{1}_{\{H^N(r)>t\}} d\mathcal{M}_r^{\ast,N} + D_s^N(t),
  \end{split}
\end{equation}
where 
 \begin{equation}\label{eq:ANst}
\begin{split}
A_s^N(t)= c(H^{N}(s)-t)^++\frac{V_{s}^{N}}{2N} \mathbf{1}_{\{H^N(s)>t\}}  
-\frac{2}{N}\int_{\delta_N}^\infty (1-e^{-Nz})\mu(dz)\int_0^s\mathbf{1}_{\{H^N(r)>t\}}\mathbf{1}_{\{V_{r}^{N}=+1\}}dr\, , 
  \end{split}
  \end{equation}
 and 
   \begin{equation}\label{eq:DNst}
\begin{split}
D_s^N(t)=& \sum_{i>0, S^N_{i}\leq s} \mathbf{1}_{\{H^N(S^N_i)>t\}} \bigg( Z^N_i-\frac{1}{N}+\inf_{S^N_{i}\le r\le s}\left( Y^N(r)-Y^N(S^N_{i})\right)\bigg)^+\\& - \int_0^s \mathbf{1}_{\{H^N(r)>t\}} \int_{0}^\infty z\, \overline{\Pi}^N(dr,dz).
  \end{split}
  \end{equation}
  The proof is organized as follows. Step 1 establishes that the sequence $\{A_s^N(t), t\ge 0 \}_{N\ge 1}$ is tight, and any limit of a converging subsequence is a. s. continuous. Step 2 shows that as $N\rightarrow +\infty$,
\begin{equation}\label{etoile}
\sup_{t\ge 0} \Big| \int_0^s \mathbf{1}_{\{H^N(r)>t\}} d\mathcal{M}_r^{\ast,N} \Big| \longrightarrow 0.
\end{equation}
Step 3 establishes that the sequence 
$ \left\{ \int_0^s \mathbf{1}_{\{H^N(r)>t\}} dB^N(r),t\ge 0 \right\}_{N\ge 1} $ is tight, and any limit of a converging subsequence is a.s. continuous. Finally step 4 shows that the sequence $\{D_s^N(t), t\ge 0 \}_{N\ge 1}$ is tight as random elements of $\mathcal{D}([0,+\infty))$. The desired result follows by combining the above arguments with Proposition \ref{TenD2} below.
 
{\sc Step 1}. The tightness of two first terms of the right--hand side of \eqref{eq:ANst} is established in the same way as in the proof of Proposition 5.7 in \cite{LePW}. 
 The sup over all $t>0$ of the absolute value of the last term is easily show to go to $0$, as $N\to\infty$, again thanks to  {\textbf (A)}. 

{\sc  Step 2}. We first note that
  \[ \sup_{t\ge 0}  \left| \int_0^s \mathbf{1}_{\{H^N(r)> t\}} d\mathcal{M}_r^{\ast,N} \right|\le \left| \mathcal{M}_s^{\ast,N} \right| + \sup_{t\ge 0}  \left|   \int_0^s \mathbf{1}_{\{H^N(r)\le t\}} d\mathcal{M}_r^{\ast,N} \right|.\]
Thanks Lemma \ref{lemMet}, it remains to prove 
  \begin{equation*}
\sup_{t\ge 0} \Big| \int_0^s \mathbf{1}_{\{H^N(r)>t\}} d\mathcal{M}_r^{\ast,N} \Big| \longrightarrow 0 \ \mbox{as} \ N\to +\infty.
\end{equation*}
 To this end, we fix $s>0$, and consider the process 
  \[G^N(t) = \int_0^s \mathbf{1}_{\{H^N(r)\le t\}} d\mathcal{M}_r^{\ast,N}.\]
Let $\mathcal{G}_t^N$ denote the $\sigma$-algebra generated by the random variables
 \[\Theta _{g_N}= \int_0^s  g_N(r) d\mathcal{M}_r^{\ast,N},\]
 where $g_N$ is bounded and $\mathcal{P}$ measurable ($\mathcal{P}$ stands for the $\sigma$-algebra of predictable subsets of $\Omega \times \R_+$) and satisfies $\{g_N(r) = 0\} \supset \{H^N(r) > t\}$. We first establish the fact that $\{G^N(t), t\ge 0\}$ is a $\mathcal{G}_t^N$-martingale. To this end, it suffices to verify that $\E[(G^N(t^{\prime})- G^N(t))\Theta _{g_N}]=0$ for $t <t^{\prime}$ and any $g_N$ as above. Indeed, it's a product of two stochastic integrals with respect to $\mathcal{M}^{\ast,N}$, with two integrands whose product is zero. Let $T>0$. In order to finally establish \eqref{etoile}, we note that 
 \begin{align*}
 \E\left(\sup_{0\le t\le T} G^N(t) \right) \le &  \left(
  \E\sup_{0\le t\le T} \left(G^N(t) \right)^2   \right)^{\frac{1}{2}}\\
  \le & \left(  4  \E \left( \int_0^s \mathbf{1}_{\{H^N(r)\le T\}} d\mathcal{M}_r^{\ast,N} \right)^2   \right)^{\frac{1}{2}}\\
   \le & \left(  \frac{8}{N}  \E  \int_0^s \mathbf{1}_{\{H^N(r)\le T\}} dr    \int_{\delta_N}^\infty z\mu(dz) \right)^{\frac{1}{2}}\\
    \le & 2\sqrt{2s} \left( \frac{1}{N} \int_{\delta_N}^\infty z\mu(dz) \right)^{\frac{1}{2}},
\end{align*}
whose tends to $0$ as $N\to\infty$, thanks to {\textbf (A)}, where we have used Cauchy Schwarz's and Doob's inequalities. The desired result follows.

{\sc  Step 3}. The tightness of $ \left\{ \int_0^s \mathbf{1}_{\{H^N(r)>t\}} dB^N(r),t\ge 0 \right\}_{N\ge 1} $ is established in the same way as in the proof of Proposition 5.7 in \cite{LePW} which we do not reproduce here. On the other hand, we will adapt the idea of this proof to treat the tightness of the sequence $\{D_s^N(t), t\ge 0 \}_{N\ge 1}$.
  
{\sc  Step 4}.  Let $\delta>0$ be a real which will eventually go to zero. Using the identity $(b-a)^+-b=-(a\wedge b)$ for $a,b>0$, we can rewrite \eqref{eq:DNst} in the following 
\[ D_s^N(t)=   X_s^N(t)+ F_s^N(t),\]  
  where
     \begin{equation*}
\begin{split}
X_s^{N}(t) =& \int_0^s  \mathbf{1}_{\{H^N(r)>t\}}  \int_{0}^\delta \left(z  -\frac{1}{N}+\inf_{r\le u\le s } Y^N(u)-Y^N(r) \right)^+\Pi^N(dr, dz)\\
& -\int_{0}^{s}  \mathbf{1}_{\{H^N(r)>t\}} \int_{0}^{\delta} z \overline{\Pi}^N(dr,dz)\\
&-\int_0^s \mathbf{1}_{\{H^N(r)>t\}}  \int_{\delta}^\infty z \wedge \left( \frac{1}{N} - \Big[\inf_{r\le u\le s } Y^N(u)-Y^N(r)  \Big]\right)\Pi^N(dr, dz) , \quad \mbox{and}\\
 F_s^N(t)=&2\int_0^{s}  \mathbf{1}_{\{H^N(r)>t\}} \mathbf{1}_{V^N_{r}=-1}dr \int_{\delta}^\infty z \mu_N(dz). 
  \end{split}
  \end{equation*}
  From Lemmas \ref{conv-muN} and \ref{kabe}, it is easy to see that
 \[   F_s^N(t) \Longrightarrow   \int_0^s  \mathbf{1}_{\{H(r)>t\}} dr \int_{\delta}^\infty z\, \mu(dz) \ \mbox{as} \ N\to +\infty. \]
 Moreover, we have 
\[ \int_0^s  \mathbf{1}_{\{H(r)>t\}} dr=  \int_t^\infty L_s(u)du,\]   
which is a.s. continuous. So according to Proposition \ref{TenD2} below, it remains only to show  the sequence $\{X_s^{N}(t), t\ge0\}_{N\ge1}$ is tight as random elements of $\mathcal{D}([0,+\infty))$. To this end, we set
\begin{equation*}
  X_s^N(t)= -X^{N,\delta,1}_s(t)+X^{N,\delta,2}_s(t), 
\end{equation*}  
where
   \begin{equation}\label{eq:XNdelta}
\begin{split}
X^{N,\delta,1}_s(t) =&\int_0^s \mathbf{1}_{\{H^N(r)>t\}}  \int_{\delta}^{1/\delta}z \wedge \left( \frac{1}{N} - \Big[\inf_{r\le u\le s } Y^N(u)-Y^N(r) \Big]\right)\Pi^N(dr, dz),\\
X^{N,\delta,2}_s(t) =& \ C^{N,\delta}_s(t) -\int_{0}^{s}  \mathbf{1}_{\{H^N(r)>t\}}  \int_{0}^{\delta} z \overline{\Pi}^N(dr,dz) 
\quad \mbox{and}\\
C^{N,\delta}_s(t)=&  \int_0^s  \mathbf{1}_{\{H^N(r)>t\}}  \int_{0}^\delta \left( z -\frac{1}{N}+\inf_{r\le u\le s } Y^N(u)-Y^N(r) \right)^+\Pi^N(dr, dz) \\&+\int_0^s \mathbf{1}_{\{H^N(r)>t\}}  \int_{1/\delta}^\infty z \wedge \left( \frac{1}{N} - \Big[\inf_{r\le u\le s } Y^N(u)-Y^N(r) \Big]\right)\Pi^N(dr, dz). 
  \end{split}
  \end{equation}
To show the tightness of the sequence $\{X_s^{N}(t), t\ge0\}_{N\ge1}$, we first show that for all $\eta>0$,
\begin{equation}\label{XN2delta}
\limsup_{N\to\infty}\P\left(\sup_{t\ge0}|X^{N,\delta,2}_s(t)| >\eta\right) \xrightarrow[\delta \rightarrow 0] \ 0.
\end{equation}
 After we will prove that the sequence $\{X_s^{N,\delta,1}(t), t\ge0\}_{N\ge1}$ is tight as random elements of $\mathcal{D}([0,+\infty))$. Finally the desired result follows by combining this with Lemma \ref{le:tight_decomp} below. 
 In order to prove \eqref{XN2delta}, we first note that 
  \begin{equation*}
 \begin{split}
\E &\left( \sup_{t \ge 0}  \left| \int_{0}^{s}  \mathbf{1}_{\{H^N(r)>t\}}  \int_{0}^{\delta} z \overline{\Pi}^N(dr,dz)   \right|  \right)\\
&\le \E   \left| \int_{0}^{s} \int_{0}^{\delta} z \overline{\Pi}^N(dr,dz)   \right| + \E \left( \sup_{t \ge 0}  \left| \int_{0}^{s}  \mathbf{1}_{\{H^N(r)\le t\}}  \int_{0}^{\delta}z \overline{\Pi}^N(dr,dz)   \right|  \right)\\
&\le \left(  \E   \left( \int_{0}^{s} \int_{0}^{\delta} z \overline{\Pi}^N(dr,dz)   \right)^{2} \right)^{\frac{1}{2}} +  \left( \E \sup_{t \ge 0}  \left| \int_{0}^{s}  \mathbf{1}_{\{H^N(r)\le t\}}  \int_{0}^{\delta} z \overline{\Pi}^N(dr,dz)   \right|^2  \right)^{\frac{1}{2}}\\
&\le   \left(  2 \E \int_0^{s} \mathbf{1}_{V^N_{r}=-1}dr \int_{\delta}^\infty z^2 \, \mu_N(dz)  \right)^{\frac{1}{2}} + \left( 8\sup_{t \ge 0}  \E \int_{0}^{s}  \mathbf{1}_{\{H^N(r)\le t\}} \mathbf{1}_{V^N_{r}=-1}dr  \int_{0}^{\delta} z^2 \, \mu_N(dz)  \right)^{\frac{1}{2}}\\
&\le   \left(  2s \int_0^{\delta} z^2 \, \mu_N(dz)   \right)^{\frac{1}{2}} +   \left(  8s \int_0^{\delta}  z^2 \, \mu_N(dz)   \right)^{\frac{1}{2}} = 3\sqrt{2s} \sqrt{\int_0^{\delta}  z^2 \, \mu_N(dz)},
\end{split}
  \end{equation*}
 where we have used Cauchy Schwarz's and Doob's inequalities in the second and the third inequalities. From Markov's inequality, we deduce that
 \begin{equation}\label{ineq:N}
\limsup_{N\rightarrow+\infty} \P \left( \sup_{t \ge 0}  \left| \int_{0}^{s}  \mathbf{1}_{\{H^N(r)>t\}}  \int_{0}^{\delta} z \overline{\Pi}^N(dr,dz)   \right| \ge \frac{\eta}{2} \right)\le \frac{6}{\eta}\sqrt{2s} \sqrt{\int_0^{\delta}  z^2 \, \mu(dz)}  \xrightarrow[\delta \rightarrow 0] \ 0,
 \end{equation}
 since as $N\to\infty$, $\int_0^\delta z^2\mu_N(dz)\to\int_0^\delta z^2\mu_(dz)$, which follows from assumption $(\mathbf{A})$ and  the following formula, which is easily established by the same computation as done in the proof of Lemma \ref{conv-muN},
 \[ \int_0^\delta z^2\mu_N(dz)=\frac{1}{N}\int_{\delta_N}^\delta z\mu(dz)+\int_{\delta_N}^\delta z^2\mu(dz)\, .\]
However, recalling \eqref{eq:XNdelta}, we have 
 \begin{equation*}
 \sup_{t\ge0}|C^{N,\delta}_s(t)| 
 \le  \int_0^s\int_{0}^\delta \left( z -\frac{1}{N}+\inf_{r\le u\le s } Y^N(u)-Y^N(r) \right)^+\Pi^N(dr, dz) +\int_0^s \int_{1/\delta}^\infty z \Pi^N(dr, dz).
\end{equation*}
Now using the Portmanteau theorem, Lemma \ref{ParZeng0} and Markov's inequality, we deduce that
 \begin{align*}
\limsup_{N\rightarrow+\infty} \P \left(  \sup_{t\ge0}|C^{N,\delta}_s(t)| > \frac{\eta}{2} \right) \le&  \P \left(  \int_0^s\int_{0}^\delta \left(z+\inf_{r\le u\le s } Y(u)-Y(r) \right)^+\Pi(dr, dz)  > \frac{\eta}{4} \right)
\\&+ \frac{8}{\eta}s\int_{1/\delta}^\infty z\mu(dz).
 \end{align*}
 We deduce from Corollary 3.5 in \cite{LiPW}  and Markov's inequality that 
 \begin{align*}
\limsup_{N\rightarrow+\infty} \P \left(  \sup_{t\ge0}|C^{N,\delta}_s(t)| > \frac{\eta}{2} \right) \le \frac{4}{\eta}C(s)\int_{0}^\delta z^2\mu(dz) + \frac{8}{\eta}s\int_{1/\delta}^\infty z\mu(dz).
 \end{align*}
 with $C(s)= (\alpha\sqrt{s/c})\vee s$ and $\alpha=e/(e-1)$, which implies that 
 \begin{equation*}
\limsup_{N\rightarrow+\infty} \P \left(  \sup_{t\ge0}|C^{N,\delta}_s(t)| > \frac{\eta}{2} \right) \le \frac{4}{\eta}C(s)\int_{0}^\delta z^2\mu(dz) + \frac{8}{\eta}s\int_{1/\delta}^\infty z\mu(dz)\xrightarrow[\delta \rightarrow 0] \ 0.
\end{equation*}
Consequently, we obtain \eqref{XN2delta} by combining this with \eqref{ineq:N} and \eqref{eq:XNdelta}. It remains to prove that the sequence $\{X_s^{N,\delta,1}(t), t\ge0\}_{N\ge1}$ is tight as random elements of $\mathcal{D}([0,+\infty))$. To this end, we show that the sequence $\{X_s^{N,\delta,1}(t), t\ge0\}_{N\ge1}$ satisfies the conditions of Proposition \ref{TenD1} below. The first condition follows easily from the fact that
\[\limsup_{N\rightarrow+\infty} \E \left(X_s^{N,\delta,1}(t)\right)\le 2s \int_{\delta}^{1/\delta} z\, \mu(dz). \] 
In order to verify the second condition, we will show that for any $T>0$, there exists $C>0, \theta>1$ such that for any $0<t<T, h>0$,
\begin{equation*}
 \E \left( \left|X_{s}^{N,\delta,1}(t+h) -X_s^{N,\delta,1}(t)\right| \left|X_s^{N,\delta,1}(t) - X_s^{N,\delta,1}(t-h)\right| \right) \le C h^\theta.
 \end{equation*}
   In order to simplify the notations below we let 
\begin{align*}
 a^{+,N}_s(t) := \mathbf{1}_{\{ t<H^N(s)\le t+h\}}, \quad \mbox{and}\quad 
 a^{-,N}_s(t) := \mathbf{1}_{\{ t-h<H^N(s)\le t\}}.
\end{align*}
An essential property, which will be crucial below, is that $a^{+,N}_s(t)\times  a^{-,N}_s(t)  = 0$. Also $(a^{+,N}_s(t))^2 = a^{+,N}_s(t)$, and similarly for $a^{-,N}$. Thus, we have  
\begin{align*}
0\le X_{t}^{N,\delta,1} - X_{t+h}^{N,\delta,1} &\le  \int_0^{s} a^{+,N}_r(t) \int_{\delta}^{\infty}  z\, {\Pi}^N(dr,dz)\\
 &= \int_0^{s} a^{+,N}_r(t) \int_{\delta}^{1/\delta}   z\,  \overline{\Pi}^N(dr,dz) + 2   \int_0^{s} a^{+,N}_r(t)  \mathbf{1}_{V^N_{r}=-1}dr \int_{\delta}^{1/\delta}    z\, \mu_N(dz),
\end{align*}
and
\begin{align*}
0\le X_{t-h}^{N,\delta,1}-X_{t}^{N,\delta,1}&\le  \int_0^{s} a^{-,N}_r(t) \int_{\delta}^{1/\delta}  z\, {\Pi}^N(dr,dz)\\
 &=\int_0^{s} a^{-,N}_r(t) \int_{\delta}^{1/\delta}   z\,  \overline{\Pi}^N(dr,dz) + 2   \int_0^{s} a^{-,N}_r(t)  \mathbf{1}_{V^N_{r}=-1}dr \int_{\delta}^{1/\delta}  z\, \mu_N(dz).
\end{align*}
Because $a^{+,N}_s(t)\times  a^{-,N}_s(t)  = 0$, the expectation of the product of 
\[ \int_0^{s} a^{-,N}_r(t) \int_{\delta}^{1/\delta}   z\,  \overline{\Pi}^N(dr,dz) \quad \mbox{with} \quad \int_0^{s} a^{-,N}_r(t) \int_{\delta}^{1/\delta}   z\,  \overline{\Pi}^N(dr,dz) \]
vanishes. We only need to estimate the expectations
\begin{equation}\label{eq:UN}
\E \left( \int_0^{s}  a^{+,N}_r(t)\int_{\delta}^{1/\delta}   z\, \overline{\Pi}^N(dr,dz) \times \int_0^{s}  a^{-,N}_r(t)  \mathbf{1}_{V^N_{r}=-1} dr \right),
\end{equation}
\begin{equation*}
\E \left( \int_0^{s}  a^{-,N}_r(t)\int_{\delta}^{1/\delta}    z\, \overline{\Pi}^N(dr,dz) \times \int_0^{s}  a^{+,N}_r(t)  \mathbf{1}_{V^N_{r}=-1} dr \right),
\end{equation*}
and
\begin{equation}\label{eq:DEUX}
\E \left( \int_0^{s} a^{+,N}_r(t)  \mathbf{1}_{V^N_{r}=-1} dr \times \int_0^{s}  a^{-,N}_r(t)  \mathbf{1}_{V^N_{r}=-1} dr \right).
\end{equation}
Since the two first equations are symmetrical, we will only estimate \eqref{eq:UN}. To this end, we use the Cauchy-Schawarz inequality and Lemma \ref{Pun} below,
\begin{align*}
&\E \left( \int_0^{s} a^{+,N}_r(t) \int_{\delta}^{1/\delta} z\, \overline{\Pi}^N(dr,dz) \times \int_0^{s}  a^{-,N}_r(t)  \mathbf{1}_{V^N_{r}=-1} dr \right) 
\\&\le  \left(  2\E  \int_0^{s}  a^{+,N}_r(t) dr \int_{\delta}^{1/\delta} z^2 \mu_N(dz) \right)^{\frac{1}{2}}  \left( \E \left( \int_0^{s}  a^{-,N}_r(t) dr \right)^2 \right)^{\frac{1}{2}} \le Ch^{3/2}. 
\end{align*}
Finally \eqref{eq:DEUX}, again from Lemma \ref{Pun} with $t^{\prime}=t+h$,
\begin{equation*}
\E \left( \int_0^{s} a^{+,N}_r(t) dr \times \int_0^{s}  a^{-,N}_r(t) dr \right) \le Ch^2.
\end{equation*}
We now conclude that the sequence $\{X_s^{N,\delta,1}(t), t\ge0\}_{N\ge1}$ is tight as random elements of $\mathcal{D}([0,+\infty))$. The desired result follows.
\end{proof}

Recall \eqref{VC}. For $K>0$, let $\tau^N_K$ be the time of the first jump of $Y^N$ of size greater than or equal to $K$. 
\begin{lemma}\label{LYNs}
Let $s, K>0$. Then there exists a constant $C$ such that for all $N\ge1$
\[ \E \left(  \sup_{0\le r < s\wedge \tau^N_K}\left|Y^N(r) \right|^2 \right) \le C. \]
\end{lemma}   
 \begin{proof}
 By combining \eqref{eq:EpsiN}, \eqref{MNs} and \eqref{VC}, it is easy to obtain that 
 \begin{equation}\label{VC2}
Y^N(s)= - \frac{V_{s}^{N}}{2N} + B^N(s) - \mathcal{M}_s^{\ast,N} +\int_0^s  \int_{\delta_N}^\infty z\, \overline{\Pi}^N(dr,dz) . 
 \end{equation}
 It follows that 
\begin{align*}
 \sup_{0\le r < s\wedge \tau^N_K}\left|Y^N(r) \right|^2 \le& \frac{4}{2N^2}+ 4  \sup_{0\le r < s}\left|\mathcal{M}_r^{\ast,N} \right|^2 + 4  \sup_{0\le r < s}\left|B^N(r) \right|^2 + 4 \left| \int_0^{ s\wedge \tau^N_K}  \int_{0}^\infty z\, \overline{\Pi}^N(dr,dz) \right|^2.
\end{align*}
From an easy adaptation of the argument used in STEP 1 in the proof of Proposition \ref{LNst-sfix}, the expectation of the last term on the right hand--side tends to $0$ as $N \to +\infty$. We now use Doob's $L^2$ inequality for martingales, which yields that there exists constant $C_2$ such that for any martingale $M$,
\[ \E \left( \sup_{0\le r \le s}\left|M_r\right|^2\right) \le C_2 \E\left(\left|M_s \right|^2 \right).\]  
Recall \eqref{DIS} and \eqref{BNS}. Hence, it suffices to notice that   
\[ \E\left(\left|B^N(s) \right|^2\right) \le 2c s, \quad   \E\left(\left|\mathcal{M}^{\ast,N}_s \right|^2 \right) \le \frac{2s}{N}\int_{\delta_N}^{\infty}z\mu(dz) , \]
\[ \left|  \int_0^{(s\wedge \tau^N_K)^-} \int_{0}^\infty z\, \overline{\Pi}^N(dr,dz)  \right| \le \sup_{r\le s} \left|  \int_0^{r}\int_{0}^K z\, \overline{\Pi}^N(du,dz)  \right|,  \]
\[\mbox{and} \quad \E \left(\left| \int_0^{s}\int_{0}^K z\, \overline{\Pi}^N(du,dz)  \right|^2 \right)  \le 2s  \int_{0}^K z^2\, \mu_N(dz). \]
The desired result follows by combining the above results.
\end{proof}

\begin{lemma}\label{Pinter}
Let $s, K>0$. Then there exists a constant $C$ such that for all $N\ge1$
\[ \sup_{t\ge0} \E [(L^N_{s\wedge \tau^N_K}(t))^2 ] \le C, \]
where $\tau^N_K$ is defined above.
\end{lemma}   
 \begin{proof}
 From Lemmas \ref{LN0INF} and \ref{LNLNINF}, we can rewrite  \eqref{2DFHNintnew} and \eqref{2DFHNintRFT1} in the following form 
 \begin{align}\label{HNsfix}
 cH^{N}(s) =&Y^N(s) -\inf_{0\le r\le s}Y^N(r) - \int_0^s\int_{0}^\infty   \left(z -\frac{1}{N}+ \inf_{r\leqslant u \leqslant s}Y^N(u)-Y^N(r) \right)^+\Pi^N(dr, dz), 
\end{align}
 \begin{equation}\label{LNsfix}
\begin{split}
 L_s^N(t)=& \Gamma^N(s,t)+ c(H^{N}(s)-t)^+ \\&+\int_0^s\int_{0}^\infty  \mathbf{1}_{\{H^N(r)>t\}} \left( z -\frac{1}{N}+ \inf_{r\leqslant u \leqslant s}Y^N(u)-Y^N(r) \right)^+\Pi^N(dr, dz),
  \end{split}
\end{equation}
where 
 \begin{equation*}
\begin{split}
\Gamma^N(s,t)= &\frac{V_{s}^{N}}{2N} \mathbf{1}_{\{H^N(s)>t\}} -  \int_0^s \mathbf{1}_{\{H^N(r)>t\}} dB^N(r) +  \int_0^s \mathbf{1}_{\{H^N(r)>t\}} d\mathcal{M}_r^{\ast,N} \\& -\int_0^s \mathbf{1}_{\{H^N(r)>t\}} \int_{0}^\infty z\, \overline{\Pi}^N(dr,dz) . 
  \end{split}
  \end{equation*}
From an adaption of the argument of proof of Lemma \ref{LYNs}, we have that there exists a constant $C$ such that for all $N\ge1$
\begin{equation}\label{AM} \sup_{t\ge0} \E \left( \sup_{0\le r \le s\wedge \tau^N_K}\left|\Gamma^N(r,t) \right|^2 \right) \le C. 
\end{equation}
We now estimate the last term on the  right of \eqref{LNsfix}. It is clear that 
 \begin{equation*}
\begin{split}
&\int_0^s\int_{0}^\infty  \mathbf{1}_{\{H^N(r)>t\}} \left( z -\frac{1}{N}+ \inf_{r\leqslant u \leqslant s}Y^N(u)-Y^N(r) \right)^+\Pi^N(dr, dz) \\&\le \int_0^s\int_{0}^\infty  \left( z -\frac{1}{N}+ \inf_{r\leqslant u \leqslant s}Y^N(u)-Y^N(r) \right)^+\Pi^N(dr, dz) \\&\le Y^N(s)-  \inf_{0\le r\le s} Y^N(r) \le 2  \sup_{0\le r\le s} |Y^N(r)|.
  \end{split}
\end{equation*}
Next we observe that 
  \begin{equation*}
\begin{split}
c(H^{N}(s)-t)^+\le cH^{N}(s)\le Y^N(s) -\inf_{0\le r\le s}Y^N(r) \le
2 \sup_{0\le r< s} |Y^N(r)|.
  \end{split}
\end{equation*}
From the last two inequalities, 
  \begin{equation*}
\begin{split}
&  \int_0^{ s\wedge \tau^N_K}\int_{\delta_N}^\infty  \mathbf{1}_{\{H^N(r)>t\}} \left( z -\frac{1}{N}+ \inf_{r\leqslant u \leqslant s}Y^N(u)-Y^N(r) \right)^+\Pi^N(dr, dz) \\&+ c(H^{N}( s\wedge \tau^N_K)-t)^+
\le
4 \sup_{0\le r < s\wedge \tau^N_K} |Y^N(r)|.
  \end{split}
\end{equation*}
The desired result follows by combining this with \eqref{AM}, \eqref{LNsfix} and Lemma \ref{LYNs}.
\end{proof}

\begin{lemma}\label{Pun}
Let $s,h,T>0$. Then there exists a constant $C$ such that for all $N\ge1$ and $0<t, t^{\prime}<T$,
\begin{align*}
&\E\left(\int_0^s a^{-,N}_r(t) dr \right) \le Ch,\\
&\E\left(\int_0^s a^{-,N}_r(t) dr \int_0^s a^{-,N}_r(t^{\prime}) dr \right) \le Ch^2.
\end{align*}
\end{lemma} 
 \begin{proof}
 We will prove the second inequality, the first one follows from the second one with $t=t^{\prime}$ and the Cauchy-Schwarz inequality. We have
 \begin{align*}
\E\left(\int_0^s a^{-,N}_r(t) dr \int_0^s a^{-,N}_r(t^{\prime}) dr \right)=& \int_{t-h}^{t} \int_{t^{\prime}-h}^{t^{\prime}} \E[L_s^N(r) L_s^N(u)] dr du\\
\le& h^2 \sup_{0\le r\le T} \E[(L_s^N(r))^2],
 \end{align*}
 where we can replace $s$ by $s\wedge \tau^N_K$. Hence the desired result follows by combining this with Lemma \ref{Pinter}. 
 \end{proof}
 
We are now ready to state the main result of this subsection. Recall \eqref{TLint} and Proposition \ref{Tanaka}.
\begin{theorem}\label{le:unifconvL}
For any $s>0$,  as $N \rightarrow \infty$, \[ \{ L_s^N(t), \ t>0\}\Longrightarrow \{L_s(t),\ t>0\} \  \mbox{in} \  \mathcal{D}([0,+\infty)), \  \mbox{locally uniformly in} \ s,\] where $L_s(t)$ is for any $s>0$, $t\ge 0$ the local time accumaled by $H$, solution of \eqref{solB}.
\end{theorem}
We shall need 
\begin{lemma}\label{le:unifcontL}
For any $T>0$, the mapping $s\mapsto L_s(t)$ is continuous, uniformly for $t\in[0,T]$.
\end{lemma}
\begin{proof}
We need to show both that for any decreasing sequence $s_n\downarrow s$, $L_{s_n}(t)-L_s(t)\downarrow0$ uniformly for $t\in[0,T]$, and for any increasing sequence $s_n\uparrow s$, $L_s(t)-L_{s_n}(t)\downarrow0$ uniformly for $t\in[0,T]$. Both statements can be proved by exactly the same argument, so we establish the first statement.

For each $n\ge1$, $t\mapsto L_{s_n}(t)-L_s(t)$ is cadlag, with only positive jumps. Consequently it is upper semi--continuous. Moreover, since $s\mapsto L_s(t)$ is continuous and increasing for any $t\in[0,T]$, $L_{s_n}(t)-L_s(t)\downarrow0$ for all $t\in[0,T]$. For any $\eps>0$, let
\[ V_n(\eps):=\{t\in[0,T];\ L_{s_n}(t)-L_s(t)<\eps\}\, .\]
Since $t\mapsto L_{s_n}(t)-L_s(t)$ is u.s.c., $V_n(\eps)$ is an open subset of $[0,T]$. However, $\cup_{n\ge1}V_n(\eps)=[0,T]$, hence there exists $N_\eps\ge1$ such that $\cup_{n\le N_\eps}V_n(\eps)=[0,T]$, and since $n\mapsto V_n(\eps)$ is increasing, $V_{N_\eps}(\eps)=[0,T]$, and for any $n\ge N_\eps$, $t\in[0,T]$, 
$L_{s_n}(t)-L_s(t)<\eps$, which establishes the result.
\end{proof}

We are now prepared to complete the

$\bf{Proof \ of \  Theorem}$  \ref{le:unifconvL} : 
For $k\ge1$, $0\le i\le [2^k\overline{s}]$, we let $s^k_i:=i2^{-k}$. Thanks to Proposition \ref{LNst-sfix}, for each pair $(i,k)$, $\{L^N_{s^k_i}(\cdot),\ N\ge1\}$ is tight in $\mathcal{D}([0,T])$. Hence along a appropriate subsequence, 
jointly for all $k\ge1$, 
\[\left(L^N_{s^k_0}(\cdot), L^N_{s^k_1}(\cdot),\ldots,L^N_{s^k_{[2^k\overline{s}]}}(\cdot))\right)\Rightarrow
\left(L_{s^k_0}(\cdot), L_{s^k_1}(\cdot),\ldots,L_{s^k_{[2^k\overline{s}]}}(\cdot))\right)\]
in $\mathcal{D}([0,T])^{[2^k\overline{s}]+1}$.
From a theorem due to Skorohod, we can and do assume that those convergences hold a.s. This means that for any $(i,k)$, 
\[ \sup_{0\le t\le T}| L^N_{s^k_i}(\lambda_N(t))- L_{s^k_i}(t)|\to 0,\]
as $N\to\infty$, where for each $N\ge1$, $\lambda_N:[0,T]\mapsto[0,T]$ is 
continuous increasing, satisfies $\lambda_N(0)=0$, $\lambda_N(T)=T$, and $\sup_{0\le t\le T}|\lambda_N(t)-t|\to0$ and $N\to\infty$. The time change $\lambda_N$ is precised in Lemma \ref{le:timechange} below. It displaces the jumps 
of $L^N_s(t)$ to those of $L_s(t)$. The $t$'s where those jumps happen do not depend upon $s$, this is why we can choose $\lambda_N$ independent 
of $(i,k)$.

Now choose $s\in[0,\overline{s}]$ arbitrary. For any $k\ge1$ arbitrarily large, there exists  $0\le i\le 2^k-1$ such that $s^k_i\le s\le s^k_{i+1}$. We have
\begin{align*}
L^N_{s^k_i}(\lambda_N(t))-L_{s^k_{i+1}}(t)&\le L^N_s(\lambda_N(t))-L_s(t)\le L^N_{s^k_{i+1}}(\lambda_N(t))-L_{s^k_{i}}(t)\\
L^N_{s^k_i}(\lambda_N(t))-L_{s^k_{i}}(t)+L_{s^k_{i}}(t)- L_{s^k_{i+1}}(t)&\le L^N_s(\lambda_N(t))-L_s(t)\le L^N_{s^k_{i+1}}(\lambda_N(t))-L_{s^k_{i+1}}(t)+L_{s^k_{i+1}}(t)-L_{s^k_{i}}(t)\,.
\end{align*}
We now choose an arbitrary $\eps>0$. Thanks to Lemma \ref{le:unifcontL}, we can choose $k$ large enough so that 
$L_{s^k_{i+1}}(t)-L_{s^k_{i}}(t)\le \eps/2$, for all $t\in[0,T]$. Hence we have
\begin{align*}
L^N_{s^k_i}(\lambda_N(t))-L_{s^k_{i}}(t)-\frac{\eps}{2}\le L^N_s(\lambda_N(t))-L_s(t)\le L^N_{s^k_{i+1}}(\lambda_N(t))-L_{s^k_{i+1}}(t)+\frac{\eps}{2}
\end{align*}
We can now choose $N$ large enough so that $\sup_{0\le i\le 2^k}\sup_{0\le t\le T}\left|L^N_{s^k_i}(\lambda_N(t))-L_{s^k_{i}}(t)\right|\le \eps/2$.
We then deduce that for such a $N$, for all $0\le s\le \overline{s}$, $0\le t\le T$, $-\eps \le L^N_s(\lambda_N(t))-L_s(t)\le \eps$, hence
\begin{align*}
\sup_{0\le s\le \overline{s}}\sup_{0\le t\le T}\left|L^N_s(\lambda_N(t))-L_s(t)\right|\le\eps\, .
\end{align*}
The result follows.
$\hfill \blacksquare$ 

\begin{lemma}\label{le:timechange}
Fix $T$ and $\overline{s}>0$, $k\ge1$ and the sequence $s^k_i=i2^{-k}$ for $0\le i\le [2^k\overline{s}]$.
There exists a random time change $\lambda_N(t)$ which is continuous and strictly increasing, such that, along an appropriate subsequence, 
\[ (\lambda_N(t),L^N_{s^k_0}(\lambda_N(t)),\ldots,L^N_{s^k_{[2^k\overline{s}]}}(\lambda_N(t)))\Rightarrow(t,L_{s^k_0}(t),\ldots,L_{s^k_{[2^k\overline{s}]}}(t)) \]
for the topology of uniform convergence on $[0,T]$.
\end{lemma}
\begin{proof} The proof will be divided in three steps. We will first define the sequence $\lambda_N$, then establish the convergence of $\lambda_N$, and finally that of $L^N_{s^k_i}(\lambda_N(t))$ for $i$ arbitrary. The fact that the above joint convergence holds along an appropriate subsequence then follows from the previous results.

{\sc Step 1}. We order the points of the measure $\Pi$ on the set $[0,\overline{s}]\times\R_+$ in decresing order of their second coordinate. This produces the sequence $\{(S_1,Z_1), (S_2,Z_2),\ldots\}$, where $Z_1>Z_2>\dots$. We associate to each $(S_i,Z_i)$ $T_i=H(S_i)$. We consider those $(T_i,Z_i)$ for which $T_i\le T$ (and delete the others). The corresponding sequence is still denoted by an abuse of notation $\{(T_1,Z_1), (T_2,Z_2), \ldots\}$.  $T_i$ is the values of $t$ at which the map $t\mapsto L_{\overline{s}}(t)$ has a jump of size $\le Z_i$. Note that for $0<s<\overline{s}$, $t\mapsto L_s(t)$ has a jump at time $T_i$ iff $S_i<s$, and $t\mapsto L_s(t)$ has jumps only at times where $t\mapsto L_{\overline{s}}(t)$ jumps. Moreover, for each $i\ge1$, there exists $s$ lare enough (possibly $>\overline{s}$) such that the jump of $t\mapsto L_{s'}(t)$ at $T_i$ is $Z_i$ for all $s'\ge s$.

Consider now the point measure $\Pi^N$, the associated $(S^N_i,Z^N_i)$, and $(T^N_i,Z^N_i)$, where $T^N_i=H^N(S^N_i)$. Again hose points are ordered in decreasing order of the $Z^N_i$'s, and only those $(T^N_i,Z^N_i)$ for which $T^N_i<T$ are taken into account. Since $\Pi^N\Rightarrow\Pi$, for each $k\ge1$, there exists $N_k$ such that for all $N\ge N_k$, the order of $(T^N_1,\ldots,T^N_k)$ is the same as that of $(T_1,\ldots,T_k)$. 

For each $k\ge1, N\ge1$ we choose as $\lambda_{N,k}$ the piecewise linear function of $t$ whose graph joins $(0,0)$, $(T_i,T^N_i)_{1\le i\le k}, (T,T)$, where the $T_i$'s are listed in increasing order. If $N\ge N_k$, then $\lambda_{N,k}(t)$ is continuous, strictly increasing and verifies $\lambda_{N,k}(0)=0$,  
$\lambda_{N,k}(T_i)=T^N_i$ for $1\le i\le k$ and $\lambda_{N,k}(T)=T$. For each $N\ge1$, we let $\lambda_N(t)=\lambda_{N,\hat{k}_N}(t)$, where
\[ \hat{k}_N=\sup\{k,\ \text{ the orders of }(T^N_i)_{1\le i\le k}\ \text{ and } (T_i)_{1\le i\le k}\ \text{ coincide}\}\,.\]

{\sc Step 2}. Since the limit $t$ of $\lambda_N(t)$ is deterministic, what we want to show is in fact that 
$\sup_{0\le t\le T}|\lambda_N(t)-t|$ tends to $0$ in probability. If $\mu(\R_+)<\infty$, then there are finitely many jumps, and the convergence $(T_i,T^N_i)\Rightarrow (T_i,T_i)$ is uniform w.r.t. $i$. The result follows. If however $\mu(\R_+)=+\infty$, then the $S_i$'s are dense in $[0,\overline{s}]$, and consequently the $T_i$'s are dense in $[0,\sup_{0\le s\le \overline{s}}H(s)\wedge T]$. For $k$ large enough, the distance between two consecutive $T_i$'s in the sequence $T_1,\ldots,T_k$ is less than $\eps$. Then for $N\ge N_k$, $\sup_{1\le i\le k}|T^N_i-T_i|\le\eps$, and the result follows. 

{\sc Step 3} We write $s$ for $s^k_i$. We assume here that all processes have been redefined in such a way that $(H^N,Y^N)\to(H,Y)$ in $\mathcal{C}(\R_+)\times\mathcal{D}(\R_+)$, and $(T^N_i,Z^N_i)_{i\ge1}\to(T_i,Z_i)_{i\ge1}$, in probability. According to \eqref{2DFHNintRFT1}, $L^N_s(t)=F^N_s(t)+ D^N_s(t)$, and combining the fact that $F^N_s(t)\to F_s(t)$ in probability uniformly in $t$ with STEP 2, we deduce that $F^N_s(\Lambda_N(t))\to F_s(t)$ uniformly in $t$ in probability. It remains to treat the term $D^N_s(\lambda_N(t))$. We now use a similar decomposition as in the proof of Lemma \ref{ParZeng0}.
\end{proof}

\begin{proposition}\label{BUT}
As $N \rightarrow \infty$, $L^N_s(H^N(s))\Rightarrow L_s(H(s)) \ \ ds$ a.e.
\end{proposition}
\begin{proof}
In order to simplify the following argument, making use of a famous theorem due to Skorohod, we may and do assume that $H^N\to H$ and $L^N\to L$ a.s.
From Proposition \ref{convHN}  $H^N(s)\to H(s)$ a.s., locally uniformly in $s$, and
according to Theorem \ref{le:unifconvL}, for all $\overline{s}>0$ and $T>0$,
\[ \sup_{0\le s\le \overline{s},\ 0\le t\le T}|L^N_s(\lambda_N(t))-L_s(t)|\to 0\quad\text{a.s.}, \ \text{as} \ N \to \infty,\]
then $ds$ a.e., $L^N_s(H^N(s))\to L_s(H(s))$ in probability. For this purpose, for any $0\le r\le \overline{s}$, we have
\[ \{t,\ t\mapsto L_r(t)\ \text{is discontinuous }\}\subset\{t,\ t\mapsto L_{\overline{s}}(t)\ \text{is discontinuous }\} \]
and the latter set is at most countable. Since $H(r)$ admit a local time, $H(r)$ 
spends zero time in such a countable set. Hence a.s., $dr$ a.e., $t\mapsto L_r(t)$ is continuous at $H(r)$.
For such an $r$, we have 
\[ | L^N_r(H^N(r))-L_r(H(r)) |\le | L^N_r(\lambda_N\circ\lambda_N^{-1}(H^N(r)))-L_r(\lambda_N^{-1}(H^N(r))) | + 
| L_r(\lambda_N^{-1}(H^N(r)))-L_r(H(r)) | \,,\]
where the random function $\lambda_N:[0,T]\mapsto[0,T]$ satisfying $\lambda_N(0)=0$, $\lambda_N(T)=T$, $\lambda_N$ is continuous and strictly increasing and $\sup_{0\le t\le T}|\lambda_N(t)-t|\to0$ is such that $\sup_{0\le t\le T}| L^N_r(\lambda_N(t))-L_r(t)|\to0$.
Define the event $\Omega_{\overline{s},T}=\{\sup_{0\le s\le \overline{s}}H(s)\le T-1\}$, 
\begin{align*}
\P\left( \Omega_{\overline{s},T}\cap\left\{| L^N_r(H^N(r))-L_r(H(r)) | >\eta\right\}\right)&\le \P\left(\sup_{0\le t\le T}|L^N_r(\lambda_N(t)t)-L_r(t)|>\eta/2\right)
\\&\quad
+\P\left(\sup_{0\le t\le T}|\lambda_N^{-1}(H^N(r))-H(r)|>1\right)
\\&\quad+\P\left(| L_r(\lambda_N^{-1}(H^N(r)))-L_r(H(r)) | >\eta/2\right)
\end{align*}
 As $N\to\infty$, the first term on the right tends to $0$ thanks to Lemma \ref{le:unifconvL}, the second term tends to $0$ since both 
 $H^N\to H$ uniformly on $[0,\overline{s}]$ and $\lambda^{-1}_N(t)\to t$ uniformly on $[0,T]$, and finally the third term tend to $0$ since $H(r)$ is a continuity point of $L_r(\cdot)$ and again $\lambda^{-1}_N(t)\to t$ uniformly on $[0,T]$. Finally, for each $\overline{s}>0$, $\P(\cup_{T>0}\Omega_{\overline{s},T})=1$.
 The result follows.
\end{proof}

\section{Convergence of the height process with interaction}\label{sec4}
In the nonlinear case where the linear drift $-bz$ is replaced by a nonlinear drift $f(z)$, the approximation of \eqref{XTPLUS} will be given by the total mass $X^{N,x}$ of a population of individuals, each of which has mass $1/N$. The initial mass is $X_0^{N,x}= [Nx]/N$, and $X^{N,x}$ follows a Markovian jump dynamics :  from its current state $k/N$,  
 \begin{equation*}\label{ZNx1}
X^{N,x} \  \mbox{jumps to}  \
\left\{
    \begin{array}{ll}
   \frac{k+\ell-1}{N} \ \mbox{at rate} \  \psi_{\delta_N}^{\prime} (N)\nu_{N}(\ell) k \ + \ N\mathbf{1}_{\{\ell=2 \}}\sum_{i=1}^{k} (f(\frac{i}{N})-f(\frac{i-1}{N}))^+,  \ \mbox{for all} \ \ell\ge2; &\\\\  \frac{k-1}{N}  \quad  \mbox{at rate} \  \psi_{\delta_N}^{\prime} (N)\nu_{N}(0) k \ + \ N\sum_{i=1}^{k} (f(\frac{i}{N})-f(\frac{i-1}{N}))^-.
    &
           \end{array}
           \right.
\end{equation*}  
  The following result is a consequence of Theorem 4.1 in \cite{DP}. 
\begin{proposition}\label{TH12}
 Suppose that Assumptions $(\bf H)$ and \eqref{hypof} are satisfied. Then, as $N\rightarrow +\infty$, $\{X_t^{N,x}, \ t\geqslant0 \}$ converges to $\{X_t^x, \ t\geqslant0 \}$ in distribution on $\mathcal{D}([0,\infty),\mathbb{R}_+)$, where $X^x$ is the unique solution of the SDE \eqref{XTPLUS}.  
\end{proposition}

The process $H^{N}$ is piecewise linear, continuous with derivative $\pm 2N$ : at any time $s\ge0,$ the rate of appearance of minima (giving rise to births, i.e. to the creation of new branches) is equal to 
\begin{equation*}
2cN^2+ 2\int_{\delta_N}^\infty(1-e^{-Nz}-Nze^{-Nz})\mu(dz)+ 2N^2 \left[ f\left(  L_s^{N}(H^{N}(s))+ \frac{1}{N} \right)-f\left(  L_s^{N}(H^{N}(s))\right) \right]^+,
\end{equation*}
and the rate of of appearance of maxima (describing deaths of branches) is equal to  
\begin{equation*}
2cN^2+2\int_{\delta_N}^\infty(e^{-Nz}-1+Nz)\mu(dz)+ 2N^2 \left[ f\left(  L_s^{N}(H^{N}(s))+ \frac{1}{N} \right)-f\left(  L_s^{N}(H^{N}(s))\right) \right]^-.
\end{equation*}
We now want to use Girsanov's theorem, in order to reduce the present model to the one studied in section 3. To this end, for $s>0,$ define 
\begin{equation}\label{PN1PNint}
\mathcal{P}_{s}^{1,N}= \int_{0}^{s}\mathbf{1}_{\{V_{r^{-}}^{N}=-1\}}dP_{r}^{N} \ \mbox{and} \ \mathcal{P}_{s}^{2,N}= \int_{0}^{s}\mathbf{1}_{\{V_{r^{-}}^{N}=+1\}}dP_{r}^{N},
\end{equation}
recall that  $P^{N}$ is a Poisson point process with intensity $2cN^2$ under the probability measure $\P,$ so that $ \mathcal{P}_{s}^{1,N}, (resp.,\mathcal{P}_{s}^{2,N})$ has the intensity 
\begin{equation*}
 \lambda_s^{1,N}= 2cN^2 \mathbf{1}_{\{V_{s^{-}}^{N}=-1\}}, \quad \mbox{resp.} \quad \lambda_s^{2,N}= 2cN^2\mathbf{1}_{\{V_{s^{-}}^{N}=+1\}}.
\end{equation*}
Recall \eqref{DIS}. We now define the collection of $\sigma$-algebras $\mathcal{F}_s^{N}:= \sigma\{H^N(r), \ 0\le r\le s\}$ and we introduce a Girsanov-Radon-Nikodym derivative 
\begin{equation}\label{UNsint}
 U_s^{N}= 1+ \int_0^s  U_{r^-}^{N} \left[(f_N^{\prime})^+ \left( L_{{r^-}}^{N}(H^N(r))\right)d\mathcal{M}_r^{1,N}  + (f_N^{\prime})^- \left( L_{{r^-}}^{N}(H^N(r))\right)d\mathcal{M}_r^{2,N}\right],
\end{equation}
with $ f_N^{\prime}(x)= N [f(x+1/N)-f(x)]. $ Under the additional assumption that $f^{\prime}$ is bounded, it is clear that $U^{N}$ is a martingale, hence $\E [ U_s^{N}]=1$ for all $s\ge0.$ In this case, we define $\widetilde{\mathbb{P}}^{N}$ as the probability such that for each $s>0,$ 
\begin{equation*}
\frac{d\widetilde{\mathbb{P}}^{N}}{d\P} \Big|_{ \mathcal{F}_s^{N}}= U_s^{N}.
\end{equation*}

It follows from Proposition \ref{Gir2} below with 
\begin{align*}
\mu_r^{1,N}&=1+\frac{1}{cN^2} \int_{\delta_N}^\infty(1-e^{-Nz}-Nze^{-Nz})\mu(dz)+ \frac{1}{cN} (f_N^{\prime})^+ \left( L_{{r^-}}^{N}(H^N(r))\right) \quad \mbox{and}\\
  \mu_r^{2,N}&=1+ \frac{1}{cN^2} \int_{\delta_N}^\infty(e^{-Nz}-1+Nz)\mu(dz)+\frac{1}{cN} (f_N^{\prime})^- \left(L_{{r^-}}^{N}(H^N(r))\right)
\end{align*}
that under $\widetilde{\mathbb{P}}^{N},$ $ \mathcal{P}_{s}^{1,N}$ ( resp., $\mathcal{P}_{s}^{2, N}$) has the intensity 
\begin{equation}\label{tauxTilde}
\begin{split}
&\left[2cN^2+2 \int_{\delta_N}^\infty(1-e^{-Nz}-Nze^{-Nz})\mu(dz)+2N(f_N^{\prime})^+ \left( L_{{r^-}}^{N}(H^N(r))\right)\right] \mathbf{1}_{\{V_{{r^-}}^{N}=-1\}},\\
\mbox{resp.} \ &\left[2cN^2+2\int_{\delta_N}^\infty(e^{-Nz}-1+Nz)\mu(dz) +2N(f_N^{\prime})^- \left( L_{{r^-}}^{N}(H^N(r))\right)\right] \mathbf{1}_{\{V_{{r^-}}^{N}=+1\}}.
\end{split}
\end{equation}
\subsection{\textit{The Case where $|f'|$ is bouned}}
We assume in this subsection that $|f'(x)|\le \beta$ for all $x\ge0$ and some $\beta>0.$ This constitutes the first step of the proof of convergence of $H^N.$ As explained at above, in this case we can use Girsanov's theorem to bring us back to the situation studied in section 3. 

\quad Recalling equations \eqref{BNS} and \eqref{VC}, we can rewrite \eqref{eq:HN} in the form 
\begin{equation*}
\begin{split}
c H^{N}(s) &= \mathcal{M}_s^{1,N}-\mathcal{M}_s^{2,N}+\mathcal{M}_s^{N}+ \epsilon^N(s) - \inf_{0\le r\le s}Y^N(r) -  \mathcal{R}^{N}(s).
 \end{split}
\end{equation*}
Moreover, from \eqref{DIS}, \eqref{PN1PNint} and \eqref{UNsint}, we have
\begin{align*}
\left[\mathcal{M}^{1,N}\right]_s&= \frac{1}{N^2} \mathcal{P}_{s}^{1,N}, \quad \left[\mathcal{M}^{2,N}\right]_s = \frac{1}{N^2} \mathcal{P}_{s}^{2,N}
\\
\langle{\mathcal{M}^{1,N}\rangle}_{s}&=2c \int_{0}^{s} \mathbf{1}_{\{V_{r}^{N}=-1\}}dr, \quad \langle{\mathcal{M}^{2,N}\rangle}_{s}=  2c \int_{0}^{s} \mathbf{1}_{\{V_{r}^{N}=+1\}}dr,\\
\left[U^{N}\right]_s = \frac{1}{N^2} \int_0^s  &\Big|U_{r^-}^{N}\Big|^2 \bigg[\bigg|(f_N^{\prime})^+ \left(L_{{r^-}}^{N}(H^{N}(r))\right) \bigg|^2  d\mathcal{P}_{r}^{1,N}\\
& + \bigg|(f_N^{\prime})^- \left( L_{{r^-}}^{N}(H^{N}(r))\right)\bigg|^2 d\mathcal{P}_{r}^{2,N}\bigg],\\
\langle{U^{N}\rangle}_{s} = 2 c \int_0^s & \Big|U_{r}^{N}\Big|^2 \bigg[\bigg| (f_N^{\prime})^+ \left( L_{{r}}^{N}(H^{N}(r))\right) \bigg|^2 \mathbf{1}_{\{V_{r}^{N}=-1\}}\\& + \bigg|(f_N^{\prime})^- \left(  L_{{r}}^{N}(H^{N}(r))\right)\bigg|^2 \mathbf{1}_{\{V_{r}^{N}=+1\}}\bigg] dr,
\\
\left[U^{N}, \mathcal{M}^{1,N}\right]_s &= \frac{1}{N^2}  \int_0^s U_{r^-}^{N} (f_N^{\prime})^+ \left( L_{{r^-}}^{N}(H^{N}(r))\right) d\mathcal{P}_{r}^{1,N},\\
\left[U^{N}, \mathcal{M}^{2,N}\right]_s &= \frac{1}{N^2} \int_0^s U_{r^-}^{N} (f_N^{\prime})^- \left( L_{{r^-}}^{N}(H^{N}(r))\right)d\mathcal{P}_{r}^{2,N},
\\
\langle{U^{N}, \mathcal{M}^{1,N}\rangle}_{s}  &= 2 c \int_0^s U_{r}^{N} (f_N^{\prime})^+ \left( L_{{r}}^{N}(H^{N}(r))\right) \mathbf{1}_{\{V_{r}^{N}=-1\}}dr,
\\
\langle{U^{N},\mathcal{M}^{2,N}\rangle}_{s}  &= 2 c \int_0^s U_{r}^{N} (f_N^{\prime})^- \left( L_{{r}}^{N}(H^{N}(r))\right)  \mathbf{1}_{\{V_{r}^{N}=+1\}}dr,
\end{align*}
while
\begin{equation*}
\left[\mathcal{M}^{1,N}, \mathcal{M}^{2,N}\right]_s= \langle{\mathcal{M}^{1,N},\mathcal{M}^{2,N}\rangle}_{s} = 0.
\end{equation*}
From Corollary \ref{MNversM}, Lemma \ref{M1NM2N} and Proposition \ref{TRIPLET}, we deduce that $\{(H^{N}, \mathcal{M}^{1,N}, \mathcal{M}^{2,N}, \mathcal{M}^{N},\mathcal{R}^{N}),\ N\ge1 \}$ is a tight sequence in $\mathcal{C}([0,\infty)) \times{(\mathcal{D}([0,\infty)))}^{4}.$ Since $f^{\prime}$ is bounded, the same is true for $ f_N^{\prime}(x)= N [f(x+1/N)-f(x)],$ uniformly with respect to $N.$ It easy to deduce from \eqref{UNsint} and Proposition \ref{TenD3} that the sequence $\{U^N, \ N\ge1\}$ is tight and as a consequence $\{(H^{N}, \mathcal{M}^{1,N}, \mathcal{M}^{2,N}, \mathcal{M}^{N},\mathcal{R}^{N}, U^N),\ N\ge1 \}$ is a tight sequence in $\mathcal{C}([0,\infty)) \times{(\mathcal{D}([0,\infty)))}^{5}.$ Therefore at least along a subsequence (but we do not distinguish between the notation for the subsequence and for the sequence),
\begin{equation*}
 (H^{N}, \mathcal{M}^{1,N}, \mathcal{M}^{2,N}, \mathcal{M}^{N},\mathcal{R}^{N}, U^N) \Rightarrow (H, \mathcal{M}^{1}, \mathcal{M}^{2}, \mathcal{M},\mathcal{R}, U)
\end{equation*}
as $N\rightarrow \infty$ in  $\mathcal{C}([0,\infty)) \times{(\mathcal{D}([0,\infty)))}^{5}.$ 

Moreover, from Lemma \ref{Ssur2int} and Proposition \ref{BUT}, we deduce that
\begin{align*}
\langle{\mathcal{M}^{1,N}\rangle}_{s}& \Rightarrow c s ,\\
\langle{\mathcal{M}^{2,N}\rangle}_{s}& \Rightarrow   c s,\\
\langle{U^{N}\rangle}_{s} & \Rightarrow  c \int_0^s  \Big|U_{r}\Big|^2 \times \bigg|f^{\prime} \left( L_{{r}}^{\Gamma}(H(r))\right) \bigg|^2  dr,
\\
\langle{U^{N}, \mathcal{M}^{1,N}\rangle}_{s}  & \Rightarrow c \int_0^s U_{r} f^{\prime +} \left(  L_{{r}}(H(r))\right) dr
\\
\langle{U^{N},\mathcal{M}^{2,N}\rangle}_{s}  & \Rightarrow c \int_0^s U_{r} f^{\prime -} \left(L_{{r}}(H(r))\right) dr.
\end{align*}
Recall Corollary \ref{MNversM}, Lemma \ref{M1NM2N} and \eqref{eq:H}. It follows from the above that Proposition \ref{TRIPLET} can be enriched as follows
\begin{proposition}\label{convU}
As $N\to \infty$, 
\begin{align*}
&\left(H^{N}, \mathcal{M}^{1,N}, \mathcal{M}^{2,N}, \mathcal{M}^{N},\mathcal{R}^{N}, U^N \right)
\Longrightarrow 
\left(H, \sqrt{c} B_{s}^{1}, \sqrt{c} B_{s}^{2}, \mathcal{M},\mathcal{R}, U\right), 
 \end{align*}
 in ${(\mathcal{C}([0,\infty)))} \times  {(\mathcal{D}([0,\infty)))}^{6},$ where $B^{1}$ and $B^{2}$ are two mutually independent standard Brownian motions. Moreover
\begin{align*}
cH(s)&= \sqrt{c} \left(B_{s}^{1}- B_{s}^{2}\right) + \int_{0}^{s}\int_{0}^{\infty}z \overline{\Pi}(dr, dz) -  \inf_{0\le r\le s}Y(r)  - \int_0^s\int_0^\infty\left(z + \inf_{r\leqslant u \leqslant s}Y(u)-Y(r) \right)^+\Pi(dr, dz),
\\ &\quad and \quad U_s=1+ \frac{1}{\sqrt{c}} \int_0^s U_{r} \bigg[ f^{\prime +} \left(L_{r}(H(r))\right) dB_r^1 + f^{\prime -} \left( L_{r}(H(r))\right) dB_r^2 \bigg].
\end{align*}
\end{proposition}
We clearly have
\begin{align*}
U_s= \exp& \bigg( \frac{1}{\sqrt{c}} \int_0^s \bigg\{  f^{\prime +} \left( L_{r}(H(r))\right) dB_r^1 + f^{\prime -} \left( L_{r}(H(r))\right) dB_r^2 \bigg\}-\frac{c}{2} \int_0^s  \bigg|f^{\prime} \left( L_{r}(H(r))\right) \bigg|^2  dr \bigg).
\end{align*}
Since $f^\prime$ is bounded $\E [ U_s]=1$ for all $s\ge0.$ Let now $\widetilde{\mathbb{P}}$ denote the probability measure such that 
\begin{equation*}
\frac{d\widetilde{\mathbb{P}}}{d\P} \Big|_{ \mathcal{F}_s}= U_s,
\end{equation*}
where $\mathcal{F}_s:= \sigma\{H(r), \ 0\le r\le s\}.$ It follows from Girsanov's theorem (see Proposition \ref{Gir1} below) that there exist two mutually independent standard $\widetilde{\mathbb{P}}$-Brownian motions $\widetilde{B}^1$ and $\widetilde{B}^2$ such that 
\begin{align*}
B_s^1&=  \frac{1}{\sqrt{c}} \int_0^s f^{\prime +} \left( L_{r}(H(r))\right) dr + \widetilde{B}_s^1\\
B_s^2&= \frac{1}{\sqrt{c}}  \int_0^s f^{\prime -} \left(  L_{r}(H(r))\right) dr + \widetilde{B}_s^2.
\end{align*}
Consequently 
\begin{align*}
 \sqrt{c} \left(B_{s}^{1}- B_{s}^{2}\right)= \sqrt{2c}B_s+ \int_0^s f^{\prime} \left(L_{r}(H(r))\right) dr, 
\end{align*}
where \[ B_s=\frac{1}{\sqrt{2}} \Big( \widetilde{B}_s^1-\widetilde{B}_s^2\Big) \] is a standard Brownian motion under $\widetilde{\mathbb{P}}.$ Consequently $H$ is a weak solution of the SDE 
\begin{align}\label{Hinteract}
cH(s)= Y(s)+ \int_0^s f^{\prime} \left(L_{r}(H(r))\right) dr-  \inf_{0\le r\le s}Y(r)  - \int_0^s\int_0^\infty\left(z + \inf_{r\leqslant u \leqslant s}Y(u)-Y(r) \right)^+\Pi(dr, dz).
\end{align}
We now deduce from Proposition \ref{convU} and Lemma 24 in \cite{EP} the main result of this subsection.
\begin{theorem}\label{fbded}
Assume that $f\in \mathcal{C}^1(\R_+)$, $f(0)=0$ and $f'$ is bounded. Then
the law of the approximate height process $H^N$, defined under $\tilde{\P}^N$ (i.e. with the Poisson processes having the intensities specified by \eqref{tauxTilde}) converges  towards the law of the height process $H$ under $\tilde{\P}$ (i.e. specified by \eqref{Hinteract}). 
\end{theorem}
\subsection{The general case $(f\in \mathcal{C}^1$ and $f' \le \theta)$}
We first note that condition \eqref{hypof} guarantees only local boundedness of $f'$. Thus, in order to make sure that Girsanov's Theorem is applicable, we use a localization procedure and associate to each $n \in (0, \infty)$ a function $f_n \in \mathcal{C}^1(\R_+)$, $f_n^{\prime}$ is uniformly continuous on $\R_+$, and  
 \begin{equation*}
f_n(x)=
\left\{
    \begin{array}{ll}
   f(x) , \quad \mbox{if} \quad 0<x\le n, &\\\\ f(n)+f'(x)(x-n) , \quad  \mbox{if} \quad  x>n.
          \end{array}
           \right.
    \end{equation*}
From this definition, it is easy to see that $f_n^{\prime}(x)= f^{\prime}(x\wedge n)$, which implies $\sup_{x\in (0,\infty)}|f_n^{\prime}(x)|=\sup_{0\le x \le n}|f^{\prime}(x)|$. 

Now we define the processes $\{U^N_n(s),U_n(s),\ s\ge0\}$ exactly as the process $U^N, U$, except that $f$ is replaced by $f_n$. Let us now state our final result.

\begin{theorem}\label{thfin}
Assume that $f\in \mathcal{C}^1(\R_+)$, $f(0)=0$ and $f'(z)\le\theta$, for some given $\theta>0$. Then, as $N\to\infty$, the law of $H^N$, specified by \eqref{DEFHNVN}
with the intensities of the Poisson processes specified by \eqref{tauxTilde} converges to the law of $H$, the solution of 
equation \eqref{Hinteract}.
\end{theorem}
\begin{proof}
We work on the probability space $(\Omega,\FF,\P)$. 
We consider the processes $H^N(r)$ and $H(r)$ restricted to an arbitrary time interval $[0,s]$. 
Suppose we have two interaction functions $f^1$ and $f^2$ which both satisfy the above assumption \ref{hypof}, and which coincide on the interval
$[0,K]$. It is then plain that the corresponding processes $H^{N,1}$ and $H^{N,2}$ (resp. $H^1$ and $H^2$) have the same law on the time interval $[0,S^N_K]$ (resp. $[0,S_K]$), where
\[ S^N_k=\inf\{s>0,\ H^{N,1}_s\vee H^{N,2}_s>K\}, \ \text {resp. } S_K=\inf\{s>0,\ H^1_s\vee H^2_s>K\,.\]

For each $m,n\ge1$, consider the event
\[A_{m,n}=\{\sup_{0\le r\le s} H(r)\le m;\ \sup_{0\le r\le s;\ 0\le t\le m}L_r(t)\le n\}.\] On the event $A_{m,n}$, 
$\sup_{0\le r\le s}L_r(H_r)\le n$. On the event $A_{m-1,n-1}$, from Proposition \ref{convHN} and Proposition \ref{BUT}, for $N$ large enough, $\sup_{0\le r\le s}H^N(r)\le m$ and $\sup_{0\le r\le s,\ 0\le t\le m}L^N_r(t)\le n$. Consequently on the event $A_{m,n}$, for such an $N$, $\sup_{0\le r\le s}L^N_r(H^N_r)\le n$, and from
Theorem \ref{fbded} with $f$ replaced by $f_n$ tells us that $H^N$ with the intensities specified by \eqref{tauxTilde}
(but with $f$ replaced by $f_n$) converges towards $H$, the weak solution of \eqref{Hinteract}, but with $f$ replaced by $f_n$.  But on the event $A_{m-1,n-1}$, and uniformly for $N$ large enough, the intensities in \eqref{tauxTilde} with $f_n$ and with $f$ coincide, and similarly for the equation \eqref{Hinteract}. Since $\cup_{m,n\ge2}A_{m,n}=\Omega$, the result follows.
\end{proof}

\section{Appendix}
In this section we recall few important notions and give some results used in this work. We do not give proofs of most of the following statements. 
\subsection{Two Girsanov Theorems}
We state two versions of the Girsanov theorem, one for the Brownian and one for the point process case. The first one can be found, e.g., in \cite{PR} and the second one combines Theorems T2 and T3 from \cite{Bre}, pages 165-166. We assume here that our probability space $(\Omega, \mathbb{P}, \mathcal{F})$ is such that $\mathcal{F}= \sigma (\cup_{t>0} \mathcal{F}_t).$    
\begin{proposition}\label{Gir1}
Let $\{B_s, \ s\ge0 \}$ be a standard $d-$dimensional Brownian motion (i.e., its coordinates are mutually independent standard scalar Brownian motions) defined on the filtered probability space $(\Omega, \mathbb{P}, \mathcal{F}).$ Let moreover $\phi$ be an $\mathcal{F}$-progressively measurable $d-$dimensional process satisfying $\int_0^s |\phi(r)|^2 dr < \infty$ for all $s\ge0.$ Let  
\begin{equation*}
U_s= \exp \left\{\int_0^s \langle \phi(r),dB_r \rangle -\frac{1}{2} \int_0^s |\phi(r)|^2 dr \right\}.
\end{equation*}
If \ $\E(U_s)=1,$ $s\ge0, $ then $\widetilde{B}_s:=B_s - \int_0^s \phi(r) dr,$ $s\ge0,$ is a standard Brownian motion under the unique probability measure $\widetilde{\mathbb{P}}$ on $(\Omega, \mathcal{F})$ which is such that $d\widetilde{\mathbb{P}}|_{ \mathcal{F}_s}/ d\P |_{ \mathcal{F}_s}= U_s,$ for all $s\ge0.$
\end{proposition} 
\begin{proposition}\label{Gir2}
Let $\{(Q_s^{(1)},..., Q_s^{(d)}),\ s\ge0 \}$ be a $d$-variate point process adapted to some filtration $\mathcal{F},$ and let $\{ \lambda_s^{(i)}, \ s\ge0\}$ be the predictable $(\P, \mathcal{F})$-intensity of $Q^{(i)},$ $1\le i\le d.$ Assume that none of the $Q^{(i)}, Q^{(j)},$ $i\neq j,$ jump simultaneously. Let  $\{ \alpha_r^{(i)}, \ r\ge0\},$ $1\le i\le d,$ be nonnegative $\mathcal{F}$-predictable processes such that for all $s\ge0$ and all  $1\le i\le d$ 
\begin{equation*}
\int_0^s  \alpha_r^{(i)}  \lambda_r^{(i)} dr < \infty \quad \mathbb{P}-a.s.
\end{equation*}
For $i=1,...,d $ and  $s\ge0$ define, $\{ T_k^{i}, \ k=1,2...\} $ denoting the jump times of  $Q^{(i)},$
\begin{equation*}
U_s^{(i)}= \left(\prod_{k\ge1:  T_k^{i}\le s} \alpha_{T_k^{i}}^{(i)} \right)\exp \left\{\int_0^s (1- \alpha_r^{(i)}) \lambda_r^{(i)} dr \right\} \quad and \quad U_s=\prod_{i=1}^d U_s^{(i)}, \quad s\ge0.
\end{equation*}
If \ $\E(U_s)=1,$ $s\ge0, $ then, for each $1\le i\le d,$ the process $Q^{(i)}$ has the $(\widetilde{\mathbb{P}}, \mathcal{F})$-intensity $\widetilde{\lambda}_s^{(i)}=\alpha_s^{(i)}  \lambda_s^{(i)},$ $s\ge0,$ where the probability measure $\widetilde{\mathbb{P}}$ is defined by $d\widetilde{\mathbb{P}}|_{ \mathcal{F}_s}/ d\P |_{ \mathcal{F}_s}= U_s,$ for all $s\ge0.$
\end{proposition} 
\subsection{Tightness criteria in \texorpdfstring{$\mathcal{D}([0, +\infty))$}{Lg}} 
We denote by $\mathcal{D}([0,\infty))$,  the space of functions from $[0,\infty)$ into $\mathbb{R}$ which are right continuous and have left limits at any $t>0$ (as usual such a function is called cˆàdlàˆg).We briefly write $\mathbb{D}$ for the space of adapted, cˆàdlàˆg stochastic processes. We shall always equip the space $\mathcal{D}([0,\infty))$ with the Skorohod topology, for the definition of which we refer the reader to Billingsley \cite{Bill} or Joffe, M\'{e}tivier \cite{JM}. 

We first state a tightness criterion, which is Theorem 13.5 from  \cite{Bill} :
\begin{proposition}\label{TenD1}
Let $(X_t^n, \ t\ge0)_{n\ge0}$ be a sequence of random elements of $\mathcal{D}([0, +\infty);\mathbb{R}^d).$ A sufficient condition of  $(X_t^n, \ t\ge0)_{n\ge0}$ to be tight is that the two conditions 1 and 2 be satisfied : 

1. For each $t,$ the sequence of random variables $(X_t^n, \ n\ge0)$ is tight in $\mathbb{R}^d.$

2. For each $T>0,$ there exists $\beta, C>0$ and $\theta>1$ such that 
$$\E (|X_{t+h}^n-X_t^n |^\beta |X_t^n-X_{t-h}^n |^\beta) \le Ch^\theta,$$ for all $0<t<T,$ $0<h <t,$ $n\ge0$.
\end{proposition}
The convergence in  $\mathcal{D}([0, +\infty);\mathbb{R}^d)$ is not additive in general. The next proposition gives a sufficient condition to have this additivity,
which is Lemma 7.1 of \cite{LePW}.
\begin{proposition}\label{TenD2}
Let ${\{X_{t}^{n}, \ t\geq0 \}}_{n\geq0}$ and ${\{Y_{t}^{n},\ t\geq0 \}}_{n\geq0}$ be two tight sequences of random elements of $\mathcal{D}([0,\infty);\mathbb{R}^d)$ such that any limit of a weakly converging sub-sequence of the sequence ${\{X_{t}^{n}, \ t\geq0 \}}_{n\geq0}$ is a.s. continuous. Then ${\{X_{t}^{n}+Y_{t}^{n}, \ t\geq0 \}}_{n\geq0}$ is tight in $\mathcal{D}([0,\infty);\mathbb{R}^d)$.
\end{proposition}
 Consider a sequence ${\{X_{t}^{n}, \ t\geq0 \}}_{n\geq1}$ of one-dimensional semi-martingales, which is such that for each $n\geq1$,
\begin{equation*}
X_{t}^{n}= X_{0}^{n}+ \int_{0}^{t}{\varphi}_s^{n}ds+ M_{t}^{n},\quad t\geq0;
\end{equation*}
where for each $n\geq1$, $M^{n}$ is a locally square-integrable martingale such that 
\begin{equation*}
\langle M^{n}{\rangle}_{t}= \int_{0}^{t}{\psi}_s^{n}ds, \quad t\geq0;
\end{equation*}
${\varphi}^{n}$ and ${\psi}^{n}$ are Borel measurable functions with values into $\mathbb{R}$ and ${\mathbb{R}}_{+}$ respectively. We define $V_{t}^{n}= X_{0}^{n}+ \int_{0}^{t}{\varphi}_{n}(X_{s}^{n})ds$. 

The following statement can be deduced from Theorem 13.4 and 16.10 of  \cite{Bill}.
\begin{proposition}\label{TenD3}
A sufficient condition for the above sequence ${\{X_{t}^{n}, t\geq0 \}}_{n\geq1}$ of semi-martingales to be tight in $\mathcal{D}([0,\infty))$ is that both
\begin{equation*}
the{~}sequence{~}of{~}r.v.'s {~}\{X_{0}^{n}, n\geq1 \} {~}is {~}tight;
\end{equation*}
and for some $p>1$,
\begin{equation*}
\forall T>0,{~} the {~}sequence {~}of {~}r.v.'s{~} \Bigg\{\int_{0}^{T}{[|{\varphi}_{n}(X_{s}^{n})|+{\psi}_{n}(X_{t}^{n})]}^{p}dt, n\geq1 \Bigg\} {~}is {~}tight.
\end{equation*}
Those conditions imply that both the bounded variation parts $\{V^{n}, n\geq1\}$ and the martingale parts $\{M^{n}, n\geq1\}$ are tight, and that the limit of any converging subsequence of $\{V^{n}\}$ is a.s. continuous.

\quad \quad If moreover, for any $T>0$, as $n\longrightarrow \infty$,
\begin{equation*}
\sup_{0\leq t \leq T} |M_{t}^{n}-M_{t^{-}}^{n}|\longrightarrow 0 {~}in {~}probability,
\end{equation*}
then any limit X of a converging subsequence of the original sequence ${\{X^{n}\}}_{n\geq1}$ is a.s. continuous.
\end{proposition}

\begin{lemma}\label{le:tight_decomp}
Let $\{X^N_t,\ t\ge0\}_{N\ge1}$ be a sequence of processes whose trajectories belong to $\mathcal{D}([0,+\infty))$
and satisfy
\begin{equation}\label{eq:bound}
 \sup_{N\ge1}\E\left(\sup_{t\ge0}|X^N_t|\right)<\infty\,.
 \end{equation}
We assume that for each $\delta>0$, there exists a decomposition $X^N_t=X^{N,\delta,1}_t+X^{N,\delta,2}_t$ such that 
$\{X^{N,\delta,1}\}_{N\ge1}$ is tight as random elements of $\mathcal{D}([0,+\infty))$, and moreover, for all $\eta>0$
\begin{equation}\label{eq:to0}
\limsup_{N\to\infty}\P\left(\sup_{t\ge0}|X^{N,\delta,2}_t| > \eta\right)\to0,\ \text{ as }\delta\to0\,.
\end{equation}
Then the sequence $\{X^N_t,\ t\ge0\}_{N\ge1}$ is tight as random elements of $\mathcal{D}([0,+\infty))$.
\end{lemma}
\begin{proof}
 We shall exploit Theorem 13.2 from \cite{Bill}. We will establish tightness in $D([0,T])$, for $T>0$ arbitrary. The moduli of continuity below are understood to be defined on the time interval $[0,T]$. Condition (i) follows from our assumption \eqref{eq:bound}.
Hence it suffices to verify (ii), namely that for each $\varepsilon,\rho>0$, there exists $\eta>0$ such that
\begin{equation}\label{eq:tight2}
\P\left(w'_{X^N}(\eta)\ge\varepsilon\right)\le\rho\, .
\end{equation}
We first note that from the definitions of $w$ (resp. $w'$) (see (7.1) (resp. (12.6)) in \cite{Bill}), for each $\eta>0$,
\begin{equation}\label{eq:ineq}
w'_{X^N}(\eta)\le w'_{X^{N,\delta,1}}(\eta)+w_{X^{N,\delta,2}}(\eta)\, .
\end{equation}
But since $w_{X^{N,\delta,2}}(\eta)\le 2\sup_{t\ge0}|X^{N,\delta,2}_t|$, for all $\eta>0$,
\begin{align*}
\limsup_{N\to\infty}\P\left(w_{X^{N,\delta,2}}(\eta)\ge\varepsilon/2\right)&\le
\limsup_{N\to\infty}\P\left(\sup_{t\ge0}|X^{N,\delta,2}_t|\ge\varepsilon/4\right).
\end{align*}
Hence from \eqref{eq:to0}, we can choose $\delta_{\varepsilon,\rho}>0$ such that
\begin{equation}\label{eq:ineq1}
\limsup_{N\to\infty}\P\left(w_{X^{N,\delta_{\varepsilon,\rho},2}}(\eta)\ge\varepsilon/2\right)\le \frac{\rho}{2} , \forall \eta>0\,.
\end{equation}
Since $\{X^{N,\delta_{\varepsilon,\rho},1}\}_{N\ge1}$ is tight, again from Theorem 13.2 from \cite{Bill},
we can choose $\eta>0$ small enough such that 
\begin{equation}\label{eq:ineq2}
\limsup_{N\to\infty}\P\left(w'_{X^{N,\delta_{\varepsilon,\rho},1}}(\eta)\ge\varepsilon/2\right)\le \frac{\rho}{2}\,.
\end{equation}
A combination of \eqref{eq:ineq}, \eqref{eq:ineq1} and \eqref{eq:ineq2} yields \eqref{eq:tight2}.

\end{proof}

\frenchspacing
\bibliographystyle{plain}

\end{document}